\documentclass[3p]{elsarticle}
\usepackage{amsmath,amssymb,amsbsy}
\usepackage{amsthm}
\usepackage{mathrsfs}
\usepackage{pgfplots}

\usepackage{multirow}

\usepackage{graphicx}
\usepackage{tikz}
\usepackage{hyperref}

\def\z{{\mathcal z}}

\def\R{{\mathbb R}}

\def\x {{\boldsymbol x}}

\def\n {{\boldsymbol n}}

\def\u {{\boldsymbol u}}
\def\v {{\boldsymbol v}}
\def\w {{\boldsymbol w}}
\def\z {{\boldsymbol z}}
\def\g {{\boldsymbol g}}
\def\f {{\boldsymbol f}}
\def\F {{\boldsymbol F}}

\newtheorem{theorem}{Theorem}
\newtheorem{proposition}{Proposition}

\newtheorem{lemma}{Lemma}
\newdefinition{definition}{\em Definition}
\newtheorem{hypothesis}{\em Hypothesis}
\newtheorem{remark}{\em Remark}
\newproof{notation}{\em \textbf{Notation}}

\begin{document}

\begin{frontmatter}
	
	\title{  Analysis of backward Euler/Spectral discretization for an evolutionary   mass and heat transfer  in porous medium}
	
		\author{Sarra Maarouf\corref{cor1}}
		\author{Driss  Yakoubi\corref{cor2}}
		
			\cortext[cor1]{Laboratoire Jacques-Louis Lions,  Universit\'e Pierre et Marie Curie, 
				boite courier 187, 4 place Jussieu, 
				75252 Paris cedex 05, France\\ 
				E-mail. sarra.maarouf@gmail.com}
		 
		  \cortext[cor2]{GIREF,
			D{\'e}partement de math{\'e}matiques
			et de statistique,
			Pavillon Vachon,\\
			1045 Avenue de la m\'edecine,
			Universit{\'e} Laval,
			Qu{\'e}bec, Canada,
			G1V 0A6.\\
			E-mail.  yakoubi@giref.ulaval.ca
		}

	\begin{abstract}
This paper presents the unsteady  Darcy's equations  coupled with two nonlinear reaction-diffusion equations,
namely  this system describes  the mass concentration and heat  transfer in porous media.
The existence and uniqueness of the solution are established for both  the variational formulation problem  and for its discrete one obtained  using  spectral  discretization. Optimal a priori  estimates are given using the Brezzi-Rappaz-Raviart theorem. We conclude by some numerical tests which are in agreement with our theoretical results.
	\end{abstract}
\begin{keyword} 
	Darcy's equation, nonlinear reaction-diffusion equations, Spectral discretization
\end{keyword}
\end{frontmatter}

  \section{Introduction}\label{Sec-Intro}
  Mass transfer is used here in a specialized sense, namely the transport of a substance which is involved as a component (constituent, species) in a fluid mixture, an example is the transport of salt  in saline water. As we shall see below, convective mass transfer is analogous to convective  heat transfer.
\par The heat transfer coupling and ground by natural convection in a porous medium saturating fluid, has received much attention in recent years due to the importance of this process that occurs in many geophysical phenomena and engineering, storage of thermal energy and recoverable systems and oil reservoirs. Several studies on this subject are performed  by Nield et al  \cite{NieldBejan}, Ingham et al \cite{ingham, pop}, Vafai \cite{vafai}, V\'adasz \cite{vadasz} and Moorthy  et al \cite{moorthy}.
\par   In the most commonly occurring circumstances,  the transport of heat and mass are not directly coupled, 
  and both  mass concentration  and heat equations  hold without change. Instead, in double-diffusive convection the coupling takes place because the density of the fluid depends on both temperature $T$ and concentration $C$. 
  In this case  Darcy's law  is written under Boussinesq approximation in unsteady case as
 \begin{alignat}2\label{darcy-F(T-C)}
 \begin{array}{rcl}
\partial_t \u+\alpha \u+ \nabla p&=& \F(T,C),\\
\nabla \cdot \u &=& 0.
\end{array}
\end{alignat} 
Indeed, for sufficiently small isobaric changes in temperature and concentration, the mixture density  $\rho$ 
depends linearly on both $T$ and $C$, and we have approximately 
\begin{align}\label{F-bous}
\F(T,C) = \rho(T,C) \, \g \, =\, \rho_0\left( 1 - \beta (T-T_0) + \beta_C(C-C_0) \right)\, \g,
\end{align}
where the subscript zero refers to a reference state, $\beta$ is the volumetric thermal
expansion coefficient, $\beta_C$ is the volumetric concentration expansion coefficient,
 $\g$  is the gravitational acceleration and $\rho_0>0$ is the  initial fluid density.
\par    In some circumstances,  there is a direct coupling. For example,  when cross-diffusion (Soret and Dufour effects) is not negligible. 
Recall that the Soret effect refers to mass flux produced by a temperature gradient, and the Dufour effect refers to heat flux produced by a concentration gradient:
    \begin{alignat}2
\partial_t T+(\u\cdot\nabla) T-\nabla\cdot(\lambda_{11}\nabla T)-\nabla\cdot(\lambda_{12}\nabla C)=h_1,\label{chaleur_T}\\
\partial_t C+(\u\cdot \nabla)C-\nabla\cdot(\lambda_{22}\nabla C)-\nabla\cdot(\lambda_{21} \nabla T)=h_2, \label{concentration_C}
 \end{alignat}
 where $\lambda_{11}, \lambda_{22}$ are respectively, the thermal and mass diffusivity and $\lambda_{12}, \lambda_{21}$ are the Dufour and Soret coefficients of the porous medium (see for example \cite{ alam, anghel,  elarabawy, narayana,  postelnicu1}).
 The variation of density with temperature and concentration leads to a combined buoyancy force. 
 The fact that the coefficients of the equation \eqref{chaleur_T} are different to those of the equation \eqref{concentration_C} leads to interesting effects, such as the oscillating flow over time in the presence of conditions for stable limits.
   \par The paper is organized as follows :
 \begin{itemize}
 \item [$\bullet$] Section 2 presents the problem setting and its analysis.
 \item [$\bullet$]  Section 3 is devoted to the description of the discrete problem using a spectral method.
\item [$\bullet$] In  Section 4, we perform the  a priori analysis of the discretization and prove optimal error estimates.
\item [$\bullet$] Finally, we describe some numerical tests in Section 5. These preliminary tests were realized with FreeFEM3D and are in agreement with our theoretical results.
 \end{itemize}
 \section{Mathematical preliminaries}\label{Sec-MathPrel}
 \subsection{Notations and definitions}\label{notation}
 \par We will use Spectral methods to approximate  the coupled equations   \eqref{darcy-F(T-C)}-\eqref{chaleur_T}-\eqref{concentration_C}. Theses methods are based on the weak formulation of the partial differential equations. In this section,  we summarize the notations and definitions needed in our analysis. 
Let $\Omega$ be a bounded open set in $\mathbb{R}^d$, $d=2$ or $3$, with a Lipschitz-continuous boundary $\partial\Omega$ divided in two parts $ \Gamma_D$ and $\Gamma_N=\partial\Omega\backslash\overline\Gamma_D$ such that

\begin{hypothesis}\label{ass1}
\begin{enumerate}[(i)]
\item \,	the intersection $\overline\Gamma_D\cap\overline\Gamma_N$ is a Lipschitz-continuous 
	submanifold of $\partial \Omega$;  
\item \, $\Gamma_D$ has a positive $(d-1)$-measure in $\partial\Omega$.
\end{enumerate}
\end{hypothesis}

\noindent The inner product and  the norm   on $L^2(\Omega)$ or $L^2(\Omega)^d$  are denoted by 
$(\cdot,\cdot)$ and    $ \| \cdot\|$  respectively. As usual, $H^s(\Omega), s\in \mathbb{R}$, denotes the real  Sobolev space 
equipped with the norm $\|\cdot\|_{H^s(\Omega)}$ and semi-norm $|\cdot|_{H^s(\Omega)}$ 
(see  for instance \cite[Chap. III and VII]{Adams}). 
For a fixed positive real number $T_f$ (which is  the final time) and a  separable Banach space $E$ equipped with the norm $\|\cdot\|_{E}$, 
we denote by $\mathcal{C}^0(0,T_f;E)$ the space of continuous functions from $[0,T_f]$ with values in $E$. For a nonnegative integer $s$, 
we also introduce the space $H^s(0,T_f;E)$ in the following way: It is the space of measurable functions on $]0,T_f[$ with values in $E$ such that the mappings: $v\mapsto\|\partial^\ell_tv\|_E, 0\leq \ell\leq s$,
are square-integrable on $]0,T_f[$. We introduce the dual space $\left(H^{\frac{1}{2}}_{00}(\Gamma_N)\right)'$ of $H^{\frac{1}{2}}_{00}(\Gamma_N)$ $\big($see \cite[Chap.1, \textsection 11]{lions} for the definition of this space$\big)$,
and denote by   $ \langle \cdot,\cdot\rangle_{\Gamma_N} $
the duality pairing between them.
 Of course the coupled equations \eqref{darcy-F(T-C)}-\eqref{chaleur_T}-\eqref{concentration_C} 
 are to be complemented with boundary and initial conditions.
\subsection{Boundary and initial conditions}\label{B-I_C}
In this section we will describe the boundary and initial conditions  for 
the coupled equations \eqref{darcy-F(T-C)}-\eqref{chaleur_T}-\eqref{concentration_C} 
which we shall consider here. We impose the homogeneous sliding boundary conditions
on whole  boundary $\partial \Omega $ on the velocity $\u$, 
while  Dirichlet--Neumann (mixed boundary conditions) on both temperature and mass $T$, $C$  are prescribed as follows:
\begin{alignat}2
&\u\cdot \n=0\hspace{0.7cm}&{\rm{on}}& ~~\partial\Omega\times]0,T_f[,\label{cond_limite2}\\
&T=T_D\qquad\quad\;\;\;&{\rm{on}}&\;\;\Gamma_D\times]0,T_f[,\qquad
\partial_{\n} T 
={\vartheta}_\sharp\hspace{0.7cm} {\rm{on}}\;\; \Gamma_N\times ]0,T_f[\label{cond_T2}\\
&C=C_D\quad &{\rm on}& ~~\Gamma_D\times]0,T_f[,
\qquad \partial_{\n}C =\Psi_\sharp \hspace{0.7cm}{\rm on}~~\Gamma_N\times ]0,T_f[
\end{alignat}
and initial conditions
\begin{equation}\label{cond-initiale}
\u(\cdot,0)=\u_0,\quad T(\cdot,0)=T_0~~{\rm and}~~ C(\cdot,0)=C_0\qquad{\rm in}~~\Omega.
\end{equation}
Where $\n$ denotes the outer normal to $\partial \Omega$ 
and the data $T_D, \vartheta_\sharp, C_D, \Psi_\sharp, \u_0, T_0$ and $C_0$ 
will be specified later.
\par In view of   equations \eqref{darcy-F(T-C)}--\eqref{concentration_C}  we need to impose  additional assumptions on the Soret and Dufour coefficients $(\lambda_{ij})_{1\le i,j\le 2}$  and of the function $\F$ in order to avoid unnecessary technicalities: \, for $ 1 \le  i,j \le 2$
	\begin{enumerate}[(a)]
\item  \, The functions $ \lambda_{ij}: \, \Omega \longrightarrow \mathbb{R} $  are nonnegative, 
continuous on $\Omega$ and bounded
 from below away from 0 and above, \, i.e. $\exists$  two  constants  $\lambda_1,\lambda_2$ such that
\begin{equation}\label{lambda}
\forall \x\in\Omega,\qquad   0 <  \lambda_1 \leq \lambda_{ij}(\x)\leq\lambda_2 < +\infty.
\end{equation}
\item  \, The following coercivity property holds
\begin{equation}\label{coercivite}
\forall \z \in \mathbb{R}^2,\quad \z \Lambda\z ^t\geq \beta |\z|^2,
\end{equation}
\begin{eqnarray*}
\mbox{where}\; 
	\Lambda=
	\left(
	\begin{array}{rcl}
		\lambda_{11}\quad\lambda_{12}\\
		\lambda_{21}\quad\lambda_{22}
	\end{array}
	\right)
	\; 
\mbox{and}	\; |\cdot| \; \mbox{is the euclidean norm in}\; \R^2.
\end{eqnarray*}
\item \, 
Due to the Boussinesq assumptions, equation \eqref{F-bous}  is not essential to our analysis, in the sense that,  
we will  consider more general  function $\F$, but under the following hypotheses:\; 
 $\F$ is continuously differentiable on $\R^2$ with bounded derivatives and there exists a couple of real numbers $(T_b,C_b)$ where $\F$ vanishes and we set 
$$\gamma=\underset{x,y\in\mathbb{R}}{\sup }\;|\nabla\F(x,y)|.$$
\end{enumerate}
So, we proceed  to the change of variable by setting ${\vartheta}=T-T_b$, $\Psi=C-C_b$ and we introduce $\f({\vartheta},\Psi)= \displaystyle \frac{1}{\gamma}\F(T,C)$. 
 This function vanishes at $(0,0)$,  is continuously differentiable on $\mathbb R^2$ 
 and the norm of its gradient is less than 1. Using the mean value Theorem, we obtain 
 \begin{equation}\label{ff}
| \f(x,y)  | \leq \left( x^2  + y^2 \right)^{\frac{1}{2}}, \qquad \forall \left(x,y \right) \in \R^2.
 \end{equation} 
Thanks to  this change, the coupled problem equations
 \eqref{darcy-F(T-C)}-\eqref{chaleur_T}-\eqref{concentration_C}  complemented  with its boundary and initial conditions can be rewritten as:

\begin{equation}\label{pblm-couple}
\left\{
\begin{array}{rlc}
 \partial_t  \u+\alpha \u + \nabla p &=\,  \gamma\f({\vartheta},\Psi)  
  \vspace{+.2cm}
\\ \vspace{+.2cm} 
 \nabla\cdot \u &= \,0\;\;
\\ \vspace{+.2cm}
  \partial_t{\vartheta}+(\u\cdot\nabla){\vartheta}-\nabla\cdot(\lambda_{11}\nabla {\vartheta})-\nabla\cdot(\lambda_{12}\nabla\Psi)
&=\,h_1 
\\ \vspace{+.2cm}
 \partial_t\Psi+(\u\cdot\nabla)\Psi-\nabla\cdot(\lambda_{21}\nabla \Psi)-\nabla\cdot(\lambda_{22}\nabla{\vartheta})
&=\,h_2
\end{array}
  \text{ in } \Omega\times ]0,T_f[
\right.
\end{equation}
\begin{equation}\label{pblm-couple-CB}
\left\{
\begin{array}{rcl}
  \u \cdot \n =0\; &\rm{on}&  \partial\Omega\times ]0,T_f[  \vspace{+.2cm}
  \\
  \vspace{+.2cm}
  {\vartheta}={\vartheta}_D= T_D-T_b \quad{\rm and}\quad \Psi=\Psi_D=C_D-C_b &{\rm on}& \Gamma_D\times]0,T_f[
  \\
  \vspace{+.2cm}
 \partial_{\n} {\vartheta} ={\vartheta}_\sharp\quad{\rm and}
 \quad
 \partial_{\n}\Psi =  \Psi_\sharp&{\rm on}&  \Gamma_N\times ]0,T_f[ 
\end{array}
\right.
\end{equation}

\begin{equation}\label{pblm-couple-CI}
\u(\cdot,0)= \u_0,\,{\vartheta}(\cdot,0)={\vartheta}_0=T_0-T_b \,{\rm{and}} \,\Psi(\cdot,0)=\Psi_0=C_0-C_b \qquad{\rm{in}}\;  \Omega.
\end{equation}

From now on, we prefer to deal with this system.
\subsection{Weak formulation}\label{Sec-VarFo}
 \par  In this section, we discuss the weak formulation of the coupled equations  \eqref{pblm-couple}. To do so,
all source functions, initial and boundary data   are assumed satisfying the following regularity assumptions  
\begin{eqnarray}\label{data3}
\begin{aligned}
&h_1,\, h_2\in L^2(0,T_f;L^2(\Omega)),   \quad {\vartheta}_\sharp, \Psi_\sharp\in L^2(0,T_f;H_{00}^{\frac{1}{2}}(\Gamma_N)'), \\
&{\vartheta}_D,  \Psi_D\in L^2(0,T_f;H^\frac{1}{2}(\Gamma_D)),\quad  \u_0\in L^2(\Omega)^d,\quad {\vartheta}_0,\Psi_0\in L^2(\Omega).
\end{aligned}
\end{eqnarray}
Since ${\vartheta}_D$ and  $ \Psi_D$ belong to $ L^2(0,T_f;H^\frac{1}{2}(\Gamma_D))$, there exist liftings  denoted by 
$\mathcal{R}({\vartheta}_D)$ and $\mathcal{R}({\Psi}_D)$ respectively, which  belong to $L^2(0,T_f;H^1(\Omega))$ such that 
\begin{equation}\label{lift-cont}
\begin{array}{rcl}
 \|\mathcal{R}({\vartheta}_D)\|_{L^2(0,T_f;H^1(\Omega))} &\leq& c_0\|{\vartheta}_D\|_{L^2(0,T_f;H^\frac{1}{2}(\Gamma_D))} 
 \\ 
\|\mathcal{R}({\Psi}_D)\|_{L^2(0,T_f;H^1(\Omega))} &\leq& c_0\|\Psi_D\|_{L^2(0,T_f;H^\frac{1}{2}(\Gamma_D))}
\end{array}
\end{equation}
where the constant $c_0>0$ depends only on $\Omega$. Also, we refer to Hopf lemma (see \cite[Chap. IV, lem. 2.3]{girault}) for the following result: for any $\varepsilon>0$ and  for all  $t\in[0,T_f]$,
\begin{equation}\label{lift-epsilon}
\begin{array}{rcl}
 \|\mathcal{R}({\vartheta}_D)(\cdot,t)\|_{H^1(\Omega)} &\leq& \varepsilon\|{\vartheta}_D(\cdot,t)\|_{H^\frac{1}{2}(\Gamma_D)} 
 \\ 
\|\mathcal{R}({\Psi}_D)(\cdot,t)\|_{H^1(\Omega)} &\leq& \varepsilon\|\Psi_D(\cdot,t)\|_{H^\frac{1}{2}(\Gamma_D)}.
\end{array}
\end{equation}
 We next consider the following space
\begin{align}
H^1_D(\Omega) \, =\, \left \{  \eta  \in H^1(\Omega),\; \eta  \,=\,0 \quad \mbox{on} \quad \Gamma_D   \right\} \label{H1*}
\end{align}
and we introduce the pressure space
$$
H^1_\diamond(\Omega)\,= \, H^1(\Omega ) \cap L^2_0(\Omega ) \,=\, \{q\in H^1(\Omega),\;\int_\Omega q\;d\x=0\}
$$
where $L^2_0(\Omega) = L^2(\Omega)/ \mathbb{R} $ the space of $L^2$ functions with vanishing mean.

\par Based on all that we have introduced above,  the  time dependent variational formulation  of  \eqref{pblm-couple} 
complemented  with its boundary and initial conditions is given as follow : \; 
Find $\left(\u,p ,\vartheta, \Psi \right)$ such that
\begin{alignat*}2
 \u  & \in H^1\left(0,T_f;L^2(\Omega)^d\right)\cap L^2\left(0,T_f;L^3(\Omega)^d\right)
\vspace{+.2cm}
 \\ \vspace{+.2cm}
  p &\in L^2\left(0,T_f;H_\diamond^1(\Omega)\right) 
  \\ \vspace{+.2cm}
 {\vartheta}, \Psi  & \in H^1\left(0,T_f;L^2(\Omega)\right)\cap L^2\left(0,T_f;H^1(\Omega)\right), 
 \end{alignat*}
 $$
 {\vartheta}={\vartheta}_D\quad{\rm and}\quad \Psi=\Psi_D\quad{\rm on}\quad\Gamma_D\times]0,T_f[
 $$
and for   {\it{a.e} } \; $ 0\leq t\leq T_f$ and for all $ \left( \v,q\right)\, \in \,  L^2(\Omega)^d \times  H_\diamond^1(\Omega),$ 
\begin{alignat}2
  \int_\Omega\partial_t\u\cdot\v\;d\x+\alpha\int_{\Omega}\u\cdot\v\;d\x+\int_{\Omega}\v\cdot\nabla p\;d\x &= \gamma\int_{\Omega}\f({\vartheta},\Psi)\cdot\v\;d\x,\label{var_darcy}
  \vspace{+.2cm}
  \\ \vspace{+.2cm}
\int_{\Omega}\v\cdot\nabla q \;d\x &=0.\label{var_div}
\end{alignat}
For all  $ \left( \eta  \,  , \Phi \right)  \, \in\,  \left(H^1_D(\Omega)\right)^2$,
\begin{alignat}2
&
 \int_\Omega \partial_t {\vartheta}\;\eta \;d\x+\int_\Omega(\u\cdot\nabla){\vartheta}\;\eta \;d\x+\int_\Omega\lambda_{11}\nabla{\vartheta}\cdot \nabla\eta \;d\x\notag
 \vspace{+.2cm}
 \\ \vspace{+.2cm}
&\hspace{1cm}+\int_\Omega\lambda_{12}\nabla\Psi\cdot \nabla\eta \;d\x=\int_\Omega h_1\;\eta \;d\x+
\langle(\lambda_{11}{\vartheta}_\sharp+\lambda_{12}\Psi_{\sharp}),\eta   \rangle_{\Gamma_N},\label{var_chaleur}
\\ \vspace{+.3cm}
&
\int_\Omega \partial_t \Psi\;\Phi\;d\x+\int_\Omega(\u\cdot\nabla)\Psi\;\Phi\;d\x+\int_\Omega\lambda_{21}\nabla{\vartheta}\cdot \nabla\Phi\;d\x\notag\\
& \hspace{1cm}
+\int_\Omega\lambda_{22}\nabla\Psi\cdot \nabla\Phi\;d\x=\int_\Omega h_2\;\Phi\;d\x+ 
\langle \left(\lambda_{22}\Psi_\sharp+\lambda_{21}{\vartheta}_\sharp\right),\Phi \rangle_{\Gamma_N}.\label{var_concentration}
\end{alignat}
\par We summarize the essential ingredients for this time dependent weak formulation to be well-posed. We first start by 
	the  inf-sup condition which is obviously to check:
\begin{alignat}2
& \forall q\in H^1(\Omega),\quad \underset{\v\in L^2(\Omega)^d}{\sup}{\frac{\displaystyle\int_{\Omega}\v\cdot\nabla q\;d\x}{\|\v\|}}=\|\nabla q\|. \label{inf-sup}
\end{alignat}
Next, we  prove  the following stability result. 
\begin{proposition}\label{Proposition1}
For all  data such that \eqref{data3} holds, 
any weak solution $(\u,{\vartheta},\Psi)$ of problem \eqref{var_darcy}--\eqref{var_concentration} satisfies the following stability estimates for all $t \in [0,T_f]$
\begin{equation}
\begin{aligned}
& \displaystyle  \|\u(\cdot,t)\|^2+\alpha  \, \|\u\|^2_{L^2(0,t;L^2(\Omega)^d)}   
 \quad \leq \,
 \|\u_0\|^2  \label{estimation-u} \\
 & \hspace{+3.5cm}+\,
\displaystyle \frac{2\gamma^2}{\alpha} \, \left (\|{\vartheta}\|^2_{L^2(0,t;L^ 2(\Omega))}
 \,+\, \|\Psi\|^2_{L^2(0,t;L^2(\Omega))}\right). \\
 \end{aligned}
 \end{equation}
 \begin{equation}\label{estimation-theta-xi}
  \begin{aligned}
 & \displaystyle \|{\vartheta}(\cdot,t)\|^2+\|\Psi(\cdot,t)\|^2+\beta \, \left( 
 \|{\vartheta}\|^2_{L^2(0,t;H^1(\Omega))}\, +\,  \|\Psi\|^2_{L^2(0,t;H^1(\Omega))} \right)  
 \\  &  \hspace{+1 cm}
 \le \, \displaystyle  \, \left(  \|{\vartheta}_0\|^2+\|\Psi_0\|^2 \right) 
 +\, \|h_1\|^2_{L^2(0,t;L^2(\Omega))}+\|h_2\|^2_{L^2(0,t;L^2(\Omega))}  \\ 
 & \hspace{+1.8cm}
 + \, 2\lambda_2\left(\|{\vartheta}_\sharp\|^2_{L^2(0,t;H^\frac{1}{2}_{00}(\Gamma_N)^{'})}+\|\Psi_\sharp\|^2_{L^2(0,t;H_{00}^\frac{1}{2}(\Gamma_N)^{'})}\right) 
 \vspace{+.2cm}
 \\ \vspace{+.2cm}
 & \hspace{+1.8cm}
 +\, \left(\lambda_2+c_0\right) \left(\|{\vartheta}_D\|^2_{H^1(0,t;H^\frac{1}{2}(\Gamma_D))}+\|\Psi_D\|^2_{H^1(0,t;H^\frac{1}{2}(\Gamma_D))}\right).
 \end{aligned}
\end{equation}
\end{proposition}

\begin{proof}
To obtain the first estimate \eqref{estimation-u}, we  take $\v =\u$ in  \eqref{var_darcy}, 
 we use successively  \eqref{var_div}, \eqref{ff} and Cauchy-Schwarz inequality, we find for all $t \in [0,T_f]$
\begin{alignat*}2
\frac{1}{2} \frac{d}{dt}  \| \u\|^2 \, +\, \alpha \| \u \|^2 \, 
 &\le  \gamma \left(\| {\vartheta}\|^2 + \| \Psi\|^2 \right)^{\frac{1}{2}}  \, \| \u\|.
\end{alignat*}
Integrating on time and using Young's inequality yield   \eqref{estimation-u}.
Next, taking test functions $\eta =\tilde {\vartheta}={\vartheta}-\mathcal R({\vartheta}_D)$  and $\Phi=\tilde \Psi = \Psi-\mathcal{R}({\Psi}_D)$ 
in \eqref{var_chaleur}-\eqref{var_concentration}, adding up the two obtained equations, we have
\begin{alignat*}2
\frac{1}{2} \frac{d}{dt} \left( \int_\Omega( |\tilde{\vartheta}|^2  
 +\, 
| \tilde\Psi|^2)\;d\x \right) 
& +\,\int_\Omega\lambda_{11} | \nabla\tilde{\vartheta}|^2\;d\x+\int_\Omega\lambda_{22} |\nabla\tilde\Psi|^2\;d\x
\\ 
&  \hspace{.5cm}
+ \int_\Omega\lambda_{12}\nabla\tilde\Psi\cdot \nabla\tilde{\vartheta}\;d\x+\int_\Omega\lambda_{21}\nabla\tilde{\vartheta}\cdot \nabla\tilde\Psi\;d\x  
\\
   & \hspace{+1cm} = \langle H_1,\tilde{\vartheta}\rangle+ \langle H_2,\tilde\Psi \rangle
\end{alignat*}
where $H_1$ and $H_2$ are defined as follow
\begin{alignat*}2
 \langle H_1, \eta  \rangle    &= \int_\Omega h_1\, \eta \, d\x 
\,+\, \langle \lambda_{11}{\vartheta}_\sharp+\lambda_{12}\Psi_{\sharp}\, ,\eta  \rangle_{\Gamma_N} \\
& \hspace{+.5cm}- \int_\Omega \partial_t \mathcal{R}({\vartheta}_D)\,\eta \,d\x
 -\, \int_\Omega(\u\cdot\nabla)\mathcal{R}({\vartheta}_D)\, \eta \,d\x \\ 
 &\hspace{+1cm} -\int_\Omega\lambda_{11}\nabla\mathcal{R}({\vartheta}_D)\cdot \nabla\eta \,d\x
\,-\, \int_\Omega\lambda_{12}\nabla\mathcal{R}({\Psi}_D)\cdot \nabla\eta \,d\x, \\  \vspace{+.2cm}
 \langle H_2,\Phi \rangle &=\int_\Omega h_2\;\Phi\;d\x 
 + \langle \lambda_{22}\Psi_\sharp+\lambda_{21}{\vartheta}_\sharp\,,\Phi \rangle_{\Gamma_N} \\
& \hspace{+.5cm}-\int_\Omega \partial_t \mathcal{R}({\Psi}_D)\,\Phi\,d\x 
- \int_\Omega(\u\cdot\nabla)\mathcal{R}({\Psi}_D) \Phi d\x
\\ 
&\hspace{+1cm}- \int_\Omega\lambda_{22}\nabla\mathcal{R}({\Psi}_D)\cdot \nabla\Phi d\x - 
\int_\Omega\lambda_{12}\nabla\mathcal{R}({\vartheta}_D)\cdot \nabla\Phi d\x.
\end{alignat*}
Hence, integrating on time $t$, using properties \eqref{lambda}-\eqref{coercivite}, 
the Cauchy-Schwarz inequality repeatedly and applying Young's inequality  combining with lifting estimates \eqref{lift-cont} and \eqref{lift-epsilon} with $\varepsilon =\frac{1}{\|\u\|_{L ^3(\Omega)^d}}$, we deduce the desired estimate.
\end{proof}
 Let us point out some remarks concerning Proposition \ref{Proposition1} and its proof.
\begin{remark} \label{remark}	
	\begin{enumerate}[(1)]
		\item 	\, In the velocity stability \eqref{estimation-u}, the dependence of unknowns on  ${\vartheta}$ and $\Psi$
	can be expressed by combining both estimates   \eqref{estimation-u} and  \eqref{estimation-theta-xi}.
	\vspace{+.2cm} 
	\item	  The dependence of the velocity in  the heat and mass stability estimate \eqref{estimation-theta-xi} does not appear explicitly. This  comes from the choice of $\varepsilon $  equal to $ \frac{1}{\|\u\|_{L ^3(\Omega)^d}}$ to bound the lifting terms. 
	\vspace{+.2cm} 
	\item \,
	 in  \cite{bergirraj}  Bernardi et al.  analyzed   problem \eqref{var_darcy}-\eqref{var_div}  provided by Dirichlet condition on the pressure. 
		They prove  the existence and uniqueness of its solution $(\u,p)$ in 
		$ H^1(0,T_f,L^2(\Omega)^d)\times L^2(0,T_f;H^1_\diamond(\Omega))$, for all data  
		${\vartheta}$ and $\Psi$ in $ L^2(0,T_f;L^2(\Omega))$  and $\u_0$ in $L^2(\Omega)^d$.
		Moreover,  this solution satisfies the following stability estimate
		\begin{alignat*}{2}
		& \|\u\|_{H^1(0,T_f; L^2(\Omega)^d)}+\|p\|_{L^2(0,T_f;H^1(\Omega))} \leq  \notag \\ 
		& \hspace{+4cm}
		c\big(\|\u_0\|+\|{\vartheta}\|_{L^2(0,T_f;L^2(\Omega))}+\|\Psi\|_{L^2(0,T_f;L^2(\Omega))}\big).
		\end{alignat*}
		\end{enumerate}
\end{remark}

Furthermore, the authors prove the following regularity result, see  \cite[Prop. 2.1]{bergirraj}
\begin{proposition}\label{prop_reg}
Let $s$  be a real number equal to $1$ if $\Omega$ is convex, to $\frac{1}{2}$ otherwise. We assume that
\begin{enumerate}[(i)]
\item ${\vartheta}$ and $\Psi$ belong to $L^2(0,T_f;H^s(\Omega)) $ and such that 
$\nabla\cdot \f({\vartheta},\Psi)$ belongs to $L^2(]0,T_f[ \times \Omega)$,
\item  the initial velocity $\u_0$ belongs to $H^s(\Omega)^d$.
\end{enumerate}
Then, the solution $(\u,p)$ of problem \eqref{var_darcy}-\eqref{var_div} belongs to $H^1(0,T_f; H^s(\Omega)^d)\times L^2(0,T_f; H^{s+1}(\Omega))$.
\end{proposition}
\noindent The following theorem give us  the  existence result.
\begin{theorem}
Assume that the domain $\Omega$ is a bounded open set of $\mathbb R ^2$, or a convex or a polyhedron of $\mathbb R^3$. 
 For all data $h_1,h_2,{\vartheta}_D,\Psi_D,{\vartheta}_\sharp,\Psi_\sharp$ and ${\vartheta}_0,\Psi_0$ satisfy \eqref{data3} 
 problems \eqref{var_darcy} to \eqref{var_concentration} has a solution.
\end{theorem}
\begin{proof} First, from Remark \ref{remark} we establish the existence of $(\u,p)$. Next, 
 considering  the application $\mathcal{F}$ which associates any $({\vartheta},\Psi)$ the solution $\u$ of $\eqref{var_darcy}-\eqref{var_div}$, then, the application
\begin{alignat*}2
({\vartheta},\Psi)\longmapsto
\left(
\begin{array}{lll}
 h_1-(\mathcal F({\vartheta},\Psi)\cdot\nabla){\vartheta}+\nabla\cdot(\lambda_{11}\nabla {\vartheta})+\nabla\cdot(\lambda_{12}\nabla\Psi)\\
 h_2-(\mathcal F({\vartheta},\Psi)\cdot\nabla)\Psi+\nabla\cdot(\lambda_{22}\nabla \Psi)+\nabla\cdot(\lambda_{21}\nabla{\vartheta})
\end{array}
\right)
\end{alignat*}
is Lipschitz-continuous on $H^1(\Omega)\times H^1(\Omega)$. Owing to density of $\mathcal D(\Omega)$ in $H^1(\Omega)$, there exists an increasing sequence $(\mathbb W_n)_n$ of  finite dimensional subspaces of $H^1(\Omega)$, such that $\cup_n \mathbb W_n$ is dense in $H^1(\Omega)$. Therefore it follows from the Cauchy$-$Lipschitz theorem that \eqref{var_chaleur}-\eqref{var_concentration} has a
unique solution $\big(({\vartheta}_n,\Psi_n)\big)_n$ in $\mathcal C^0(0,T_f;\mathbb W_n)^2$. 
Since $\big(({\vartheta}_n,\Psi_n)\big)_n$ is bounded using \eqref{estimation-theta-xi}, there exists a  subsequence $\big(({\vartheta}_n,\Psi_n)\big)_n$ which converges weakly to $({\vartheta},\Psi)$ in $H^1(\Omega)\times H^1(\Omega)$. 
Thanks to Proposition \ref{prop_reg}, the mapping  $\mathcal F$ is continuous from $H^1(\Omega)\times H^1(\Omega)$ into $H^s(\Omega)^d$.  
Combining with compactness embedding of $H^s(\Omega)$ in $L^3(\Omega)$ (see Amrouche et al. \cite[prop. 3.7]{ABDG}), there exists a subsequence still denoted by $\big(({\vartheta}_n,\Psi_n)\big)_n$  which converges weakly to $({\vartheta},\Psi)$ in $H^1(\Omega)\times H^1(\Omega)$ and such that the sequence $(\mathcal F({\vartheta}_n,\Psi_n))_n$ converges strongly to $\mathcal F({\vartheta},\Psi)$ in $L^3(\Omega)^d$. Then, the sequences $(\mathcal F({\vartheta}_n,\Psi_n)\cdot\nabla {\vartheta}_n)_n$ and $(\mathcal F({\vartheta}_n,\Psi_n)\cdot\nabla \Psi_n)_n$ respectively converge to $\mathcal F({\vartheta},\Psi)\cdot\nabla{\vartheta}$ and $\mathcal F({\vartheta},\Psi)\cdot\nabla\Psi$. 
Finally we deduce that $({\vartheta},\Psi) $ belongs to $\mathcal C^0(0,T_f,H^1(\Omega))\times\mathcal C^0(0,T_f,H^1(\Omega))$ and  is solution of problem \eqref{var_darcy}--\eqref{var_concentration}.
\end{proof}
The uniqueness result requires additional assumptions. 
 \begin{proposition}
 Assume that problem \eqref{var_darcy}--\eqref{var_concentration} has a solution $(\u,p,{\vartheta},\Psi)$ such that 
 $({\vartheta},\Psi)$ belongs to $L^\infty(\Omega)\times L^\infty(\Omega)$ and  there exists a nonnegative constant $\mu$ such that 
 \begin{equation}\label{cond_unicit_uptc} 
 \|{\vartheta}\|_{L^\infty(\Omega)}+\|\Psi\|_{L^\infty(\Omega)}\leq \mu.
\end{equation}
Then, this solution is unique.
 \end{proposition}
\begin{proof} Let $(\u_1,p_1,{\vartheta}_1,\Psi_1)$ et $(\u_2,p_2,{\vartheta}_2,\Psi_2)$ be two  solutions of problem
	\eqref{var_darcy}--\eqref{var_concentration} such that $({\vartheta}_1,\Psi_1)$ belongs to  $L^\infty(\Omega)\times L^\infty(\Omega)$ and satisfies \eqref{cond_unicit_uptc}. We set
$$\u=\u_1-\u_2,\quad p=p_1-p_2,\quad {\vartheta}={\vartheta}_1-{\vartheta}_2, \quad \Psi=\Psi_1-\Psi_2.$$
we proceed in two steps.\\
1) Taking $\eta ={\vartheta}$ and $\Phi=\Psi$ in $\eqref{var_chaleur}$ and $\eqref{var_concentration}$, this yields
\begin{alignat*}2
&\hspace{0cm}\frac{1}{2}\frac{d}{dt}(\|{\vartheta}\|^2+\|\Psi\|^2)+\beta(\|\nabla{\vartheta}\|^2+\|\nabla \Psi\|^2)\\
&\hspace{3cm}\leq 
\|\u\|(\|{\vartheta}_1\|_{L^{\infty}(\Omega)}\|\nabla {\vartheta}\|+\|\Psi_1\|_{L^{\infty}(\Omega)}\|\nabla \Psi\|).
\end{alignat*}
By using Young's inequality, we obtain 
\begin{alignat*}2
&\hspace{-1cm}\frac{1}{2}\frac{d}{dt}(\|{\vartheta}\|^2+\|\Psi\|^2)\leq 
\frac{1}{2\beta}\|\u\|^2(\|{\vartheta}_1\|^2_{L^{\infty}(\Omega)}+\|\Psi_1\|^2_{L^{\infty}(\Omega)}),
\end{alignat*}
since ${\vartheta}_1$ et $\Psi_1$ satisfy \eqref{cond_unicit_uptc}
\begin{alignat*}2
&\hspace{0cm}\frac{d}{dt}\big(\|{\vartheta}\|^2+\|\Psi\|^2\big)\leq  \frac{\mu^2}{\beta}
\|\u\|^2.
\end{alignat*}
Integrating between $0$ and  $t$, we deduce that
\begin{alignat}2
&\hspace{0cm}\|{\vartheta}(\cdot,t)\|^2+\|\Psi(\cdot,t)\|^2\leq   \frac{\mu^2}{\beta}
\|\u\|^2_{L^2(0,t;L^2(\Omega)^d)}.\label{theta_xi_u}
\end{alignat}
2) Next we take $\v=\u$ in \eqref{var_darcy},  this yields 
$$\frac{1}{2}\frac{d}{dt}\|\u\|^2+\alpha\|\u\|^2\leq \gamma\big(\|{\vartheta}\|+\|\Psi\|\big)\|\u\|.$$
Once again using Young's inequality combining with \eqref{theta_xi_u} give us
$$\frac{d}{dt}\|\u\|^2\leq \frac{\mu^2\gamma^2}{\alpha\beta}\|\u\|^2_{L^2(0,t;L^2(\Omega)^d)}.$$
It follows from Gronwall's lemma that $\u$ is equal to 0. So, we derive from \eqref{theta_xi_u} that ${\vartheta}= \Psi=0 $. This concludes the proof.
\end{proof}
\section{The time semi-discrete problem}\label{section_SemiDiscret}
Now, we consider the time discretization of the coupled problem \eqref{var_darcy}--\eqref{var_chaleur} introduced in the last section. In order to make its analysis, we  first introduce a partition of the interval $[0,T_f]$ into subintervals $[t_{m-1},t_m]$, $1\leq m\leq M$, such that 
$0=t_0<t_1<\cdot\cdot\cdot<t_M=T_f$. We denote by $\tau_m$ the time step $t _m-t_{m-1}$, 
by $\tau$ the $M$-tuple $(\tau_1,...,\tau_M)$ and by $|\tau|$ the maximum of the $\tau_m$, $1\leq m\leq M$. We assume that  
$$
\displaystyle  \max_{m} \frac{\tau_m}{\tau_{m-1}}\leq\sigma_\tau,
\quad \mbox{where}\; \sigma_{\tau} \; \mbox{depends on the discretization.}
$$
We denote by $\u^m, p^m, {\vartheta}^m$ and $\Psi^m$  the approximate solutions at time $t_m$, 
and we assume  that all data are continuous in time, see later.

\noindent The implicit backward Euler's scheme applied to  \eqref{var_darcy}---\eqref{var_chaleur} results in: 
\begin{itemize}
\item {\bf Find} $ \; 
 (\u^m,p^m,{\vartheta}^m,\Psi^m)\in L^3(\Omega)^d\times H^1_\diamond(\Omega)\times H^1(\Omega)\times H^1(\Omega), 
$  such that
$$ {\vartheta}^m={\vartheta}_D^m,\quad{\rm and}\quad \Psi^m=\Psi_D^m\quad{\rm on}\quad\Gamma_D$$
\item {\bf Initialization:}  \; $ \u^0=\u_0, \; p^0=p_0, \; {\vartheta}^0={\vartheta}_0$  and $ \Psi^0=\Psi_0$  in \; $\Omega $
\item {\bf Darcy Step:} \;  for all $(\v,q) \in  L^2(\Omega)^d \times  H^1_\diamond(\Omega) $
\begin{alignat}{2}
\int_\Omega \frac{\u^m-\u^{m-1}}{\tau_m}\cdot\v d\x&+
\alpha\int_\Omega \u^m\cdot\v d\x+\int_\Omega \v\cdot\nabla p^m d\x \notag 
\\ & \hspace{+1cm}=\gamma 
\int_\Omega \f({\vartheta}^{m},\Psi^{m})\cdot\v d\x, \label{prblm-semi-discret1} \\
 \int_\Omega \u^m\cdot\nabla q\;d\x&=0.\notag 
  \end{alignat}
 \item {\bf Heat Step:} \; $ \forall\eta \in H^1_D(\Omega)$
\begin{alignat}{2}
& \int_\Omega \frac{{\vartheta}^m-{\vartheta}^{m-1}}{\tau_m}\eta \;d\x+\int_\Omega (\u^m\cdot\nabla){\vartheta}^m\eta \;d\x
 +\int_\Omega \lambda_{11} \nabla{\vartheta}^m\cdot\nabla \eta \;d\x 
 \notag \\
 &\hspace{+.1cm} +\int_\Omega \lambda_{12} \nabla\Psi^m\cdot\nabla \eta \;d\x
 =\int_\Omega h_1^m\eta \;d\x+\langle(\lambda_{11}{\vartheta}^m_\sharp+\lambda_{12}\Psi^m_\sharp),\eta \rangle_{\Gamma_N}.
 \label{prblm-semi-discret2}
\end{alignat}
 \item {\bf Concentration Step:} \; $ \forall\Phi \in H^1_D(\Omega)$
 \begin{alignat}{2}
& \int_\Omega \frac{\Psi^m-\Psi^{m-1}}{\tau_m}\Phi\;d\x+\int_\Omega (\u^m\cdot\nabla)\Psi^m\Phi\;d\x
+\int_\Omega \lambda_{22} \nabla\Psi^m\cdot\nabla \Phi\;d\x\notag\\
&\hspace{.1cm} +\int_\Omega \lambda_{21} \nabla{\vartheta}^m\cdot\nabla \Phi\;d\x
=\int_\Omega h_2^m\Phi\;d\x+ \langle(\lambda_{22}\Psi^m_\sharp+\lambda_{21}{\vartheta}^m_\sharp),\Phi\rangle_{\Gamma_N}.
\label{prblm-semi-discret3}
\end{alignat}
\end{itemize}
\begin{remark}
\begin{enumerate}[(a)]
\item In this scheme, it is commonly accepted that the   initial pressure $p^0$ can be taken 
equal to  the atmospheric pressure  which is in $H^1(\Omega)$. 
\item It can be noted that this problem makes sense since all $h_1$, $h_2$, ${\vartheta}_D$, ${\vartheta}_\sharp$, $\Psi_D$ and $\Psi_\sharp$ are continuous on time.
\end{enumerate}
\end{remark} 
Proving its well-posedness relies on rather different arguments as previously.
\begin{proposition}\label{prop-3-2}
For all continuous data functions
\begin{alignat*}2
& h_1,h_2\in \mathcal C^0(0,T_f;L^2(\Omega)), \quad
 {\vartheta}_D,\Psi_D\in \mathcal C^0(0,T_f;H^\frac{1}{2}(\Gamma_D)),\\
&{\vartheta}_\sharp,\Psi_\sharp\in \mathcal C^0(0,T_f; H^\frac{1}{2}_{00}(\Gamma_N)^{'}),\quad\u_0\in L^2(\Omega)^d\quad {\rm and}\quad{\vartheta}_0,\Psi_0\in L^2(\Omega),
\end{alignat*}
scheme \eqref{prblm-semi-discret1}-\eqref{prblm-semi-discret2} and \eqref{prblm-semi-discret3} complemented with 
initialization step has a solution $(\u^m,p^m,{\vartheta}^m,\Psi^m)$ in $L^3(\Omega)^d\times H^1_\diamond(\Omega)\times H^1(\Omega)\times H^1(\Omega)$. Moreover, this solution satisfies the following estimates for all $m$, $0\leq m\leq M$
\begin{equation}\label{stab-um-pm}
\|\u^m\|^2+\sum_{j=1}^m \tau_j \|\nabla p^j\|^2\leq c\Big(\sum_{j=1}^m \tau_j\|{\vartheta}^j\|^2+\sum_{j=1}^m \tau_j\|\Psi^j\|^2\Big)+\|\u_0\|^2,
\end{equation}
and
\begin{alignat}2
 \|{\vartheta}^m\|^2&+\|\Psi^m\|^2+  \beta\Big(\sum_{j=1}^m \tau_j\|{\vartheta}^j\|_{H^1(\Omega)}^2+\sum_{j=1}^m \tau_j\|\Psi^j\|_{H^1(\Omega)}^2\Big)\notag\\
& \leq \|{\vartheta}_0\|^2+\|\Psi_0\|^2+c\Big(\sum_{j=1}^m \tau_j\|h_1^j\|^2+\sum_{j=1}^m \tau_j\|h_2^j\|^2
\notag\\
& +\sum_{j=1}^m \tau_j\|{\vartheta}_\sharp^j\|_{H^{\frac{1}{2}}_{00}(\Gamma_N)'}^2
+ \sum_{j=1}^m \tau_j\|\Psi_\sharp^j\|_{H^{\frac{1}{2}}_{00}(\Gamma_N)'}^2
\notag\\ 
& +\sum_{j=1}^m  \tau_j\|{\vartheta}_D^j\|_{H^{\frac{1}{2}}(\Gamma_D)}^2+\sum_{j=1}^m \tau_j\|\Psi_D^j\|_{H^{\frac{1}{2}}(\Gamma_D)}^2
\label{stab-thetam-xim} \\
& +\sum_{j=1}^m\tau_j\Big\|\frac{{\vartheta}_D^j-{\vartheta}_D^{j-1}}{\tau_j}\Big\|^2_{H^\frac{1}{2}(\Gamma_D)}
+\sum_{j=1}^m\tau_j\Big\|\frac{\Psi_D^j-\Psi_D^{j-1}}{\tau_m}\Big\|^2_{H^\frac{1}{2}(\Gamma_D)}\Big).\notag 
\end{alignat}
\end{proposition}
\begin{proof} We skip the proof of existence of $(\u^m,p^m,{\vartheta}^m,\Psi^m)$ which is rather standard and simpler than in 
	Section \ref{Sec-MathPrel} and we prove only the estimates \eqref{stab-um-pm} and \eqref{stab-thetam-xim}.
To prove the first one, we take $\v=\u^m$ in  the first step \eqref{prblm-semi-discret1}. Standard arguments give
$$
\|\u^m\|^2\leq c\Big(\sum_{j=1}^m \tau_j\|{\vartheta}^j\|^2+\sum_{j=1}^m \tau_j\|\Psi^j\|^2\Big)+\|\u_0\|^2.
$$
In the second step  \eqref{prblm-semi-discret1}, we take $\v=\nabla p^m$, this yields
$$
\sum_{j=1}^m \tau_j \|\nabla p^j\|^2\leq c\Big(\sum_{j=1}^m \tau_j\|{\vartheta}^j\|^2+\sum_{j=1}^m \tau_j\|\Psi^j\|^2\Big)+\|\u_0\|^2.
$$
 To establish \eqref{stab-thetam-xim}, we perform the following change of variable   in \eqref{prblm-semi-discret2}-\eqref{prblm-semi-discret3}
 $$ 
 \tilde{\vartheta}^m = {\vartheta}^m-\mathcal R({\vartheta}^m_D) \qquad \mbox{and} \qquad \tilde\Psi^m=\Psi^m-\mathcal R(\Psi^m_D)
 $$
 Then,  joins together the obtained  equations and taking $\left( \eta ,\Phi\right)=\left( \tilde{\vartheta}^m, \tilde\Psi^m\right)$, we have
\begin{alignat*}2
&\frac{1}{2\tau_m}\big( 
\|\tilde{\vartheta}^m\|^2 - \|\tilde{\vartheta}^{m-1}\|^2+ \|\tilde{\vartheta}^m-\tilde{\vartheta}^{m-1}\|^2  \\ 
& \hspace{+3cm}+ \|\tilde\Psi^m\|^2-\|\tilde\Psi^{m-1}\|^2+\|\tilde\Psi^m-\tilde\Psi^{m-1}\|^2\big)  \\
&
+\int_\Omega \lambda_{11} (\nabla\tilde{\vartheta}^m)^2d\x + \int_\Omega \lambda_{22} (\nabla\tilde\Psi^m)^2d\x  \\ 
& \hspace{3cm}
+\int_\Omega \lambda_{12} \nabla\tilde\Psi^m\cdot\nabla \tilde{\vartheta}^m d\x
+\int_\Omega \lambda_{21} \nabla\tilde{\vartheta}^m\cdot\nabla \tilde\Psi^m d\x   \\
&\hspace{0.5cm}
=\int_\Omega h_1^m\tilde{\vartheta}^m d\x  + \langle (\lambda_{11}{\vartheta}^m_\sharp   +\lambda_{12}\Psi^m_\sharp),\tilde{\vartheta}^m\rangle_{\Gamma_N} \\ 
&\hspace{0.7cm}
+\int_\Omega h_2^m\tilde\Psi^m d\x+\langle(\lambda_{21}\Psi^m_\sharp+\lambda_{22}{\vartheta}^m_\sharp),\tilde\Psi^m\rangle_{\Gamma_N}\\
& \hspace{0.9cm}-
 \int_\Omega \frac{\mathcal{R}({\vartheta}_D^m)-\mathcal{R}({\vartheta}_D^{m-1})}{\tau_m}\tilde{\vartheta}^m d\x
 - \int_\Omega (\u^m\cdot\nabla)\mathcal{R}({\vartheta}_D^m)\tilde{\vartheta}^m d\x\\ 
 & \hspace{1.1cm}
 -\int_\Omega \lambda_{11} \nabla\mathcal{R}({\vartheta}_D^m)\cdot\nabla \tilde{\vartheta}^m d\x
+\int_\Omega \lambda_{12} \nabla\mathcal{R}({\Psi}_D^m)\cdot\nabla \tilde{\vartheta}^m d\x
\\
&\hspace{1.3cm}
-\int_\Omega \frac{\mathcal{R}({\Psi}_D^m)-\mathcal{R}({\Psi}_D^{m-1})}{\tau_m}\tilde\Psi^m d\x 
-\int_\Omega (\u^m\cdot\nabla)\mathcal{R}({\Psi}_D^m)\tilde\Psi^m d\x \\ 
&\hspace{1.5cm}
-\int_\Omega \lambda_{22} \nabla\mathcal{R}({\Psi}_D^m)\cdot\nabla \tilde\Psi^m d\x
-\int_\Omega \lambda_{21} \nabla\mathcal{R}({\vartheta}_D^m)\cdot\nabla \tilde\Psi^m d\x.
\end{alignat*}
Next we use the coercivity property \eqref{coercivite} and Cauchy-Schwarz inequality, we obtain
\begin{alignat*}3
&\frac{1}{2\tau_m}\Big( 
\|\tilde{\vartheta}^m\|^2 - \|\tilde{\vartheta}^{m-1}\|^2+ \|\tilde{\vartheta}^m-\tilde{\vartheta}^{m-1}\|^2  \\ 
& \hspace{+4cm}+ \|\tilde\Psi^m\|^2-\|\tilde\Psi^{m-1}\|^2+\|\tilde\Psi^m-\tilde\Psi^{m-1}\|^2\Big)  \\
&\hspace{.5cm}+\beta(\|\nabla\tilde{\vartheta}^m\|^2+\|\nabla\tilde\Psi^m\|^2)\hspace{0cm}
\\
&\hspace{.5cm}
\leq \Big[
\|h_1^m\|+\lambda_2\left(\|{\vartheta}^m_\sharp\|_{H^\frac{1}{2}_{00}(\Gamma_N)'}+\|\Psi^m_\sharp\|_{H^\frac{1}{2}_{00}(\Gamma_N)'}\right) \\ 
&\hspace{4.5cm} 
+\big\|\frac{\mathcal{R}({\vartheta}_D^m)-\mathcal{R}({\vartheta}_D^{m-1})}{\tau_m}\big\|+\lambda_2\|\nabla\Psi_D^m\|\Big]
\|\tilde{\vartheta}^m \big\|\\
&\hspace{0.5cm}
+\Big[
\|h_2^m\|+ \lambda_2\left(\|{\vartheta}^m_\sharp\|_{H^\frac{1}{2}_{00}(\Gamma_N)'} + \|\Psi^m_\sharp\|_{H^\frac{1}{2}_{00}(\Gamma_N)'}\right) \\ 
&\hspace{4.5cm} 
+\big\|\frac{\mathcal{R}({\Psi}_D^m)-\mathcal{R}({\Psi}_D^{m-1})}{\tau_m}\big\|+\lambda_2\|\nabla{\vartheta}_D^m\|
\Big]
\|\tilde\Psi^m\|\\
&\hspace{0.5cm}+\left(\|\u^m\|_{L^3(\Omega)^d}\|\mathcal{R}({\vartheta}_D^m)\|_{L^6(\Omega)}+\lambda_2\|\nabla\mathcal{R}({\vartheta}_D^m)\|\right)
\|\nabla\tilde{\vartheta}^m\|\\
&\hspace{0.5cm}
+\left(\|\u^m\|_{L^3(\Omega)^d}\|\mathcal{R}({\Psi}_D^m)\|_{L^6(\Omega)}
+\lambda_2\|\nabla\mathcal{R}({\Psi}_D^m)\|\right)\|\nabla\tilde\Psi^m\|.
\end{alignat*}
First, using Poincar\'e and Young's inequalities combining with \eqref{lift-epsilon}. Next,
multiplying by $\tau_m$ and summing over  $m$ give the desired estimate.
\end{proof}
\section{The fully discrete problem}\label{section:FDP}
In order to apply the spectral method to our problem, we first 
 assume that the domain $\Omega$ is the square or the cube $]-1,1[^d$, $d=2$ or $3$,
and  that all data $h_1$, $h_2$, ${\vartheta}_D$, ${\vartheta}_N$, $\xi_D$ and $\xi_N$ are continuous 
on $\overline\Omega \times [0,T_f]$ and on  $\partial\Omega \times [0,T_f] $.
We next  introduce, for each  nonnegative $n$,  the space $\mathbb P_n(\Omega)$ 
of restrictions to $\Omega$ of polynomials with $d$ variables and degree
with respect to each variable lower or equal to $ n$.
Concerning  the discrete spaces, let us first denote by $N$  a fixed positive integer. Then, we define the following
standard  working discrete spaces:
\begin{equation}\label{space_DisVP}
\left\{
\begin{array}{rcl}
\mathbb{X}_N=  \mathbb P_N(\Omega)^d, 
&&  \mathbb{Y}_N=  \mathbb P_N(\Omega), 
\\  \mathbb{Y}_N^{\diamond}=  \mathbb P_N(\Omega) \cap H^1_{\diamond}(\Omega),  &&
\; \mathbb{Y}_N^{\star}=  \mathbb P_N(\Omega) \cap H^1_{\star}(\Omega).
\end{array}
\right.
\end{equation}
   We recall that there exist a unique set of $N+1$ nodes $\left(\xi_j\right)_{0\leq j\leq N}$, with 
   $\left(\xi_0,\xi_N\right)=\left(-1,1\right)$  and a unique set of $N+1$ weights $\left(\rho_j\right)_{0\leq j\leq N}$, such that the following equality holds
\begin{equation}\label{GL}
\forall \phi   \in \mathbb{P}_{2N-1}(-1,1),\quad\int_{-1}^1 \phi (\zeta)\;d\zeta=\sum_{i=0}^N\phi (\xi_i)\rho_i.
\end{equation}
We also recall the following property, which is useful 
throughout this paper, see  for instance Bernardi et al.  \cite{BMR}
\begin{equation}\label{GL:Maj}
\forall \phi _N\in \mathbb{P}_N(-1,1),\quad \|\phi  _N\|^2_{L^2(-1,1)}\leq \sum_{i=0}^N\phi  _N^2(\xi_i)\rho_i\leq 3\|\phi _N\|_{L^2(-1,1)}^2.
\end{equation}
\par As standard in spectral methods, we introduce the grid  $\Xi $ defined by:
\begin{equation}\label{grid}
\Xi = 
\left \{
\begin{array}{rcl}
\left\{ (\xi_i,\xi_j); \quad 0\le i,j \le N\right\}  \;\qquad \qquad \qquad &\mbox{in the case} \; d=2, \\
\{ (\xi_i,\xi_j,\xi_k); \quad  0\le i,j,k \le N\}  \qquad \qquad  &\mbox{in the case} \; d=3.
\end{array}
\right.
\end{equation}
We denote by $\mathcal I_N$ the Lagrange interpolation operator at the nodes of the grid $\Xi$ with values in $ \mathbb P_N(\Omega)$, 
and by  $i^{\Gamma_D}_N$ 
the Lagrange interpolation operator  at the nodes of $\Xi \cap \overline\Gamma_D$
with values in the space of traces of functions in $\mathbb{P}_N(\Omega)$ on $\Gamma_D$.
\par Finally, we introduce the discrete product which is a scalar product on $\mathbb{P}_N(\Omega)$ from \eqref{GL:Maj}, defined for all continuous functions $u$ and $v$ on $\overline{\Omega}$ by
\begin{eqnarray}\label{PSD}
 (u,v)_N = 
 \left\{
 \begin {array}{rcl}
           \displaystyle \sum_{i=0}^N\sum_{j=0}^N u(\xi_i,\xi_j) \, v(\xi_i,\xi_j)\rho_i\rho_j\hspace{2.cm}&{\rm{if}}&\;d=2\\
          \displaystyle \sum_{i=0}^N\sum_{j=0}^N\sum_{k=0}^Nu (\xi_i,\xi_j,\xi_k) \, v(\xi_i,\xi_j,\xi_k)\rho_i\rho_j\rho_k\hspace{0.5cm}&{\rm{if}}&\;d=3.
           \end{array}
\right.
\end{eqnarray}
Also, on each edge or face $\Gamma_\ell$ of $\Omega$, we define a discrete product 
( for example  if $\Gamma_\ell =\{-1\}\times]-1,1[^{d-1}$ )
\begin{eqnarray}\label{PSD_bord}
(u,v)_N^{\Gamma_{\ell}}=
\left\{\begin{array}{rcl}
 \hspace{-0cm}
\displaystyle  \sum_{j=0}^N u(\xi_0,\xi_j)\, v(\xi_0,\xi_j)\rho_j  \hspace{2.5cm}&{\rm{if}}&\; d=2,\\
\hspace{-0cm}
\displaystyle \sum_{j=0}^N\sum_{k=0}^N u(\xi_0,\xi_j,\xi_k)\, v(\xi_0,\xi_j,\xi_k)\rho_j \rho_k\hspace{1.cm} &{\rm{if}}&\; d=3.
         \end{array}
\right.
\end{eqnarray}
A global product on $\Gamma_N$ is then defined by adding all ones
$$ \big(u,v\big)_M^{\Gamma_N}=\displaystyle\sum_{\ell\in \mathcal L} \big(u,v\big)_M^{\Gamma_\ell} $$
where $\mathcal L$ stands for the set of indices $\ell$ such that $\Gamma_\ell$ is contained in $\Gamma_N$.
 Then, for all  continuous data functions $h_1,h_2$   on $\overline\Omega\times[0,T_f]$
  and ${\vartheta}_N, \Psi_N$ on $\overline\Gamma_N\times[0,T_f]$, the discrete problem is  obtained from            \eqref{prblm-semi-discret1}--\eqref{prblm-semi-discret3} and its  initialization step:
\begin{itemize}
\item{\bf Find}:\;  $ \left(\u_N^m, p^m_N,{\vartheta}^m_N,\Psi^m_N\right) \in \mathbb{X}_N\times \mathbb{Y}^\diamond_N\times \mathbb{Y}_N\times \mathbb{Y}_N,$ such that
\item {\bf Initialization:}\;  $ \u^0_N=\mathcal I_N \u_0,\; 
 {\vartheta}_N^0=\mathcal I_N {\vartheta}_0$ \; and \;
$ \Psi_N^0=\mathcal I_N\Psi_0 $ \; in   $\Omega$
\item {\bf Darcy Step:}\; for all   $ \v_N\in \mathbb{X}_N$  and for all $q_N \in \mathbb{Y}^\diamond_N$
\begin{alignat}{2}
 \left(\frac{\u^m_N-\u_N^{m-1}}{\tau_m},\v_N\right)_N + 
\alpha\left(\u^m_N,\v_N\right)_N + 
\left(\v_N,\nabla p^m_N\right)_N
&= \gamma\left(\f({\vartheta}^{m}_N,\Psi_N^{m}),\v_N\right)_N \notag\\
\left(\u^m_N,\nabla q_N\right)_N &=0.  \label{discrete_pbl1}
\end{alignat}
\item {\bf Heat Step:}\; for all    $ \eta _N\in \mathbb{Y}^\star_N$
\begin{alignat}{2}
&{\vartheta}^m_N=\mathcal I^{\Gamma_D}_N {\vartheta}_D^m\quad {\rm on}\; \Gamma_D\notag\\
&\left(\frac{{\vartheta}^m_N-{\vartheta}_N^{m-1}}{\tau_m},\eta _N\right)_N
+\left((\u^m_N\cdot\nabla ){\vartheta}^m_N,\eta _N\right)_N +
\left(\lambda_{11}\nabla{\vartheta}^m_N,\nabla\eta _N\right)_N  \notag \\
&\hspace{.2cm}+
\left(\lambda_{12}\nabla\Psi^m_N,\nabla\eta _N\right)_N 
=\left(h_1^m,\eta _N\right)_N +\left((\lambda_{11}{\vartheta}^m_\sharp+\lambda_{12}\Psi_\sharp^m),\eta _N\right)_N^{\Gamma_N}.\label{discrete_pbl2}
\end{alignat}
\item {\bf Concentration Step:}\; for all    $ \Phi_N\in \mathbb{Y}^\star_N$
\begin{alignat}{2}
&\Psi^m_N=\mathcal I^{\Gamma_D}_N \Psi_D^m\quad {\rm on}\; \Gamma_D\notag\\
& \left(\frac{\Psi^m_N-\Psi_N^{m-1}}{\tau_m},\Phi_N\right)_N
+\left((\u^m_N\cdot\nabla )\Psi^m_N,\Phi_N\right)_N
+\left(\lambda_{22}\nabla\Psi^m_N,\nabla\Phi_N\right)_N  \notag\\
&\hspace{.2cm}
+\left(\lambda_{21}\nabla{\vartheta}^m_N,\nabla\Phi_N\right)_N
=\left(h_2^m,\eta _N\right)_N+\left((\lambda_{22}\Psi^m_\sharp+\lambda_{21}{\vartheta}^m_\sharp),\Phi_N\right)_N^{\Gamma_N}.\label{discrete_pbl3}
\end{alignat}
\end{itemize}
 As in Section \ref{Sec-MathPrel},  similar arguments can be 
used to prove 
the existence result of the fully discrete scheme  \eqref{discrete_pbl1}--\eqref{discrete_pbl3}: (we skip its proof)
\begin{proposition}
For all  continuous functions $h_1,h_2$ on $\overline \Omega\times[0,T_f]$, ${\vartheta}_D,\Psi_D$ on $\overline\Gamma_D\times[0,T_f]$ and ${\vartheta}_\sharp,\Psi_\sharp$ on $\overline\Gamma_N\times[0,\tilde{T}]$, the discrete scheme \eqref{discrete_pbl1}-\eqref{discrete_pbl2}-\eqref{discrete_pbl3}  complemented  with the initialization step has a solution.
\end{proposition}
\section{A priori error estimate}\label{Apriori}
In this section, we derive {\it a priori} estimates for the velocity, pressure, heat and mass unknowns which make consistent.
The remainder of our analysis  is somewhat technical. In fact, following the  prior  works  by Brezzi-Rappaz-Raviart \cite{BRR}, 
we  here  develop  a  thorough  theory  of  { \it a priori}  error estimates  of spectral solutions of \eqref{discrete_pbl1}--\eqref{discrete_pbl3}. 
We  first define the linear operator $\mathcal T$, which associates with
 any data $(\f,\u_0)$ in $L^2(0,T_f;L^2(\Omega)^d)\times L^2(\Omega)^d$, the solution 
$\tilde U=(\u,p)$ of the following problem: 
\qquad $ \u|_{t=0}= \u_0\quad {\rm{in}} \; \Omega$  and for all 
$(\v,q)  \in L^2(\Omega)^d \times H_\diamond^1(\Omega)$
\begin{alignat}2\label{darcy_lin}
 \int_\Omega\partial_t \u\cdot\v\;d\x + \alpha\int_{\Omega}\u\cdot\v\;d\x + \int_\Omega\v\cdot\nabla p\;d\x &=\gamma \int_\Omega \f\cdot \v\;d\x,\\
 \int_\Omega \u\cdot \nabla q\;d\x&= 0. \notag
 \end{alignat}
We introduce   the linear operator $\mathcal L$  which associates with  any data 
$h_i, \omega_{iD},\omega_{i\sharp}$ and $\omega_{i0}, \, i,j=1,2$, $ i\neq j$  the solution $\omega_i$ of the following problem
\begin{equation}\label{chaleur-concentration-lin}
\left\{
\begin{aligned}
 &\partial_t \omega_i-\nabla\cdot(\lambda_{ii}\nabla\omega_i)-\nabla\cdot
 (\lambda_{ij}\nabla\omega_j)=h_i \qquad{ \rm in}\quad \Omega \times ]0,T_f[\\
 &\omega_i=\omega_{iD}\quad{\rm on}\quad \Gamma_D \quad{\rm and} \quad 
 \partial_\n \omega_i =\omega_{i\sharp}\; \;\quad \quad {\rm on}\quad \Gamma_N \times ]0,T_f[\\
 &\omega_i(\cdot,0)=\omega_{i0}\hspace{+4.8cm}{ \rm in}\quad \Omega\times ]0,T_f[.
\end{aligned}
\right .
\end{equation}
To simplify, we set the following space ({\it global space} )
\begin{equation*}
\begin{aligned}
\mathcal X= \big(H^1(0,T_f;L^2(\Omega)^d)&\cap L^2(0,T_f;L^3(\Omega)^d)\big)
\times  L^2(0,T_f;H_\diamond^1(\Omega))\\
&\hspace{0cm} \times\left( H^1(0,T_f;L^2(\Omega))\cap L^2(0,T_f;H^1(\Omega)) \right) \\ 
&\hspace{0cm} \times \left(H^1(0,T_f;L^2(\Omega))\cap L^2(0,T_f;H^1(\Omega))\right)
\end{aligned}
\end{equation*}
and denote by $\mathcal{E}$ the space of endomorphisms of $\mathcal{X}$. Setting $U = \left(\u,p,{\vartheta},\Psi\right)$.
We observe that  $(\omega_1,\omega_2)$ coincides with $({\vartheta},\Psi)$  such that
 problem \eqref{var_darcy}--\eqref{var_concentration} can equivalently be written as:  
\begin{eqnarray}
 &\mathcal{F}(U)=
 U -\left(
\begin{array}{lll}
\mathcal{T}&  0\\
0&\mathcal{L}\\
\end{array}
\right)
\left(
\begin{array}{l}
\mathcal G^1(U)\\
\mathcal G^2(U)\\
\end{array}
\right)=0
 \end{eqnarray}
 \begin{eqnarray*}
\mbox{where} \qquad              \left(
\begin{array}{l}
\mathcal{G}^1(U)\\
\mathcal{G}^2(U)\\
\end{array}
\right)=\left(
\begin{array}{l}
\f(\omega_1,\omega_2)\\
\left(h_i- \left(\u\cdot\nabla \right)\omega_i,\, \omega_{iD},\, \omega_{i\sharp},\, \omega_{i0}\right)\\
\end{array}
\right),\quad i=1,2.
              \end{eqnarray*}
\subsection{About the time discretization}\label{SubSec-TD}
With each family of values $(v^m)_{0\leq m\leq M}$, we associate the function $v_\tau$ which is affine on each interval $[t_{m-1},t_m]$, $1\leq m\leq M$, and equal to $v^m$ at $t= t_m$, $1\leq m\leq M$. For each continuous function $v$  on $[0,T_f]$, we also introduce the function $\pi_\tau^+v$ which is constant,
equal to $v(t_m) $   on each interval $]t_{m-1},t_m]$, $1\leq m\leq M$. 
We first consider the semi-discrete operator $\mathcal T_\tau$ which associated with any data $\f$ in $\mathcal C^0(0,T;L^2(\Omega)^d)$ and $\u_0$ in $L^2(\Omega)^d$, $\mathcal T_\tau(\f,\u_0)$ is equal to
$(\u_\tau,p_\tau)$ associated with  $(\u^m,p^m)$ solutions of the following semi-discrete problem 
\\

 \noindent Find $\left( \u^m, p^m \right)_{0\leq m\leq M}$ in $(L^2(\Omega)^d)^{M+1} \times H^1(\Omega)^M$ 
 such that  for all $  \v \in L^2(\Omega)^d$
\begin{alignat}{2}\label{darcy-semi-dis-lin}
&
 \u^0=\u_0 \quad{\rm in}\quad \Omega,\qquad 
{\rm and\; for\; a.\; e.\; }m,\;  1\leq m\leq M,\notag\\
& \int_\Omega \frac{\u^m-\u^{m-1}}{\tau_m}\cdot\v\;d\x+
\alpha\int_\Omega \u^m\cdot\v\;d\x+\int_\Omega \v\cdot\nabla p^m\;d\x
=\gamma
\int_\Omega \f^m\cdot\v\;d\x.
\end{alignat}
Next, let $\mathcal L_\tau$ be the semi-discrete operator defined  as follow:  \, For any data
\begin{align*}
\left(h_i,\omega_{iD},\omega_{i\sharp},\omega_{i0}\right)  \quad \mbox{in}\quad 
 &\mathcal C^0\left(0,T_f;L^2(\Omega)\right) 
\times \mathcal C^0\left(0,T_f;H^\frac{1}{2}(\Gamma_D)\right) \\ 
&\times \mathcal C^0\left(0,T_f;H^\frac{1}{2}_{00}(\Gamma_N)^{'}\right)
\times L^2(\Omega),\quad i=1,2
\end{align*}
$\mathcal L_\tau\left(h_i,\omega_{iD},\omega_{i\sharp},\omega_{i0}\right)$ 
is equal to $\omega_{i\tau}$ associated with    $\omega_i^m$  solution of  
\begin{equation}\label{chaleur-concentration-semi-discret}
\left\{
\begin{aligned}
 \omega_i^0=\omega_{i0}\; \mbox{in}\; \Omega, & \quad
 \omega_i^m=\omega_{iD}^m\; \mbox{on}\;
 \Gamma_D
 \qquad \mbox{and   for all}\; \eta \in H^1_D(\Omega),\\
 \int_{\Omega}\frac{\omega_i^m-\omega_i^{m-1}}{\tau_m}\eta \,d\x
 &+ \int_\Omega \lambda_{ii}\nabla\omega^m_i\cdot\nabla\eta ~d\x+\int_\Omega\lambda_{ij}
 \nabla\omega_j^m \cdot\nabla\eta \,d\x \\
 &\hspace{-0,5cm}\displaystyle=\int_\Omega h_i^m\eta ~d\x+\langle\lambda_{ii}\omega^m_{i\sharp}+\lambda_{ij}\omega^m_{j\sharp},\eta \rangle_{\Gamma_N},\quad i,j=1,2, i\neq j .
\end{aligned}
\right .
\end{equation}
With these definitions, the basic properties of the semi-discrete operators $\mathcal T_\tau$ and $\mathcal L_\tau$ are given in  the next proposition.
\begin{proposition}\label{prop_L-T_Oper}
The operators $\mathcal T_\tau$ and $\mathcal{L}_\tau$ satisfy the following three properties
\begin{enumerate}[(i)]
\item { \bf Stability:} \quad 
 $ \forall \left( \f,h_1, h_2\right)  \in L^2(0,T_f;L^2(\Omega)^d)\times L^2(0,T_f;L^2(\Omega))^2$, we have
\begin{eqnarray*}
\|\mathcal T_\tau(\f,0)\|_{\mathcal{C}^0(0,T_f;L^2(\Omega)^d)\times L^2(0,T_f;H^1(\Omega))}\leq c\|\pi_\tau^+ \f\|_{L^2(0,T_f;L^2(\Omega)^d)},\\
\sum_{i=1}^2\|\mathcal L_\tau(h_i,0,0,0)\|_{\mathcal{C}^0(0,T_f;H^1(\Omega))}\leq c\sum_{i=1}^2\|\pi_\tau^+ h_i\|_{L^2(0,T_f;L^2(\Omega))}.
\end{eqnarray*}
\item {\bf A priori error estimate:} We assume that the solution $\u$ belongs to $H^2(0,T_f;L^2(\Omega)^d)$ and that $\omega_i$, $i=1,2$, belong to $H^2(0,T_f;H^1(\Omega))$, then

\begin{equation} \label{T-Ttau} 
\begin{aligned}
\|(\mathcal{T}-\mathcal{T}_{\tau})(\f,\u_0)\|&_{\mathcal{C}^0(0,T_f;L^2(\Omega)^d)\times L^2(0,T_f;H^1(\Omega))}  \\ 
& \leq c|\tau| \, \| \u\|_{H^2(0,T;L^2(\Omega)^d)}, 
\end{aligned}
\end{equation}
\begin{equation}\label{L-Ltau}
\begin{aligned}
\sum_{i=1}^2\|(\mathcal L-\mathcal L_\tau)(h_i,\omega_{iD},\omega_{i\sharp},\omega_{i0})\|&_{\mathcal C^0(0,T_f;H^1(\Omega))} 
\\ 
& \leq c |\tau|\sum_{i=1}^2\|\omega_i\|_{H^2(0,T_f;H^1(\Omega))}.
\end{aligned}
\end{equation}
\item\quad {\bf Convergence:} \quad \\
  $ \forall \left( \f,h_1, h_2\right)  \in L^2(0,T_f;L^2(\Omega)^d)\times L^2(0,T_f;L^2(\Omega))^2$,
\begin{alignat*}2
&\lim_{|\tau|\rightarrow0} \biggl( 
\|(\mathcal{T}-\mathcal{T}_{\tau})(\f,0)\|_{\mathcal C^0(0,T_f;L^2(\Omega)^d)\times  L^2(0,T_f;H^1(\Omega))} \\ 
& \hspace{+3cm}+\;
\|(\mathcal{L}-\mathcal{L}_{\tau})(h_i,0,0,0)\|_{\mathcal C^0(0,T_f;H^1(\Omega))} 
\biggl)\; =\; 0.
\end{alignat*}
\end{enumerate}
\end{proposition} 
\begin{proof}
The stability property is a direct consequence of Proposition $ \ref{prop-3-2} $ in the linear case. So, we only prove the property of a priori error estimate. First, we estimate the error on the velocity between the semi-discrete problem $\eqref{darcy-semi-dis-lin}$ and the continuous one $\eqref{darcy_lin}$. 
The error equation is obtained by subtracting $(\ref{darcy-semi-dis-lin})$ from $(\ref{darcy_lin})$ at time $t_m$.  for a. e. $ m,\;  1\leq m\leq M,$ \; 
$\forall \left( \v, q\right)  \in L^2(\Omega)^d \, \times \,H^1_0(\Omega)  $ 
\begin{alignat*}{2}
 \int_\Omega \frac{e_\u^m-e_\u^{m-1}}{\tau_m}\cdot\v\;d\x +
\alpha\int_\Omega e_\u^m\cdot\v\;d\x &+\int_\Omega \v\cdot\nabla (p(\cdot,t_m)-p^m)\;d\x
\\
& =\gamma
\int_\Omega \varepsilon^m_\u\cdot\v\;d\x,
\\
& \int_\Omega e_\u^m\cdot\nabla q\;d\x=0,
\end{alignat*}
where the sequence $(e_\u^m)_m$ is defined by $e^m_\u=\u(\cdot,t_m)-\u^m$ and satisfied $
e^0_\u=0$ and the consistence error $\varepsilon^m_\u$ is given by
$$ \displaystyle \varepsilon^m_\u= \frac{\u(\cdot,t_m)-\u(\cdot,t_{m-1})}{\tau_m}-\partial_t\u(\cdot,t_m).$$
\noindent As the same arguments in \cite[Prop. 3.2, Cor. 3.1]{bergirraj} we obtain
\eqref{T-Ttau} and   \eqref{L-Ltau}.
\end{proof} 
\subsection{About the space discretization}\label{SubSec-SD}
 Henceforth, we denote $v_{N\tau}$  the function which is affine on each interval $[t_{m-1} , t_m]$ and
equal to $v^m_{N}$ at each time $t_m , 0 \leq m \leq M$. We also define the discrete
operator $\mathcal T_{N\tau}$ as follow: 

\par For any data $\f$ in $\mathcal C^0(0,T_f; L^2(\Omega)^d)$ and $\u_0$ in $L^2(\Omega)^d$, $\mathcal T_{N\tau}(\f,\u_0) $ is equal to $(\u_{N\tau},p_{N\tau})$  which interpolates $(\u_N^m,p_N^m)$ solutions of: \qquad For all $  \left( \v_N,\, q_N \right) \, \in \,  \mathbb{X}_N \times \, \mathbb{Y}_N$
\begin{equation}\label{50}
\left \{
\begin{array}{rclr}
&&\u_N^0 = \mathcal I_N  \u_0 \quad\mbox{in} \; \Omega  \\
&& \left(\displaystyle\frac{\u_N^m-\u_N^{m-1}}{\tau_m},\v_N\right)_N+ 
 \alpha\left(\u_N^m,\v_N\right)_N+\left(\v_N,\nabla p_N^m\right)_N 
=\,\gamma  \left(\f^m,\v_N\right) ,\\ 
&&\left(\u_N^m,\nabla q_N\right)_N=0.
\end{array}
\right .
\end{equation}
Finally, we denote by $ \mathcal L_{N\tau}$ the operator which  associates with any data
\begin{alignat*}{2}
\left(h_i,\,\omega_{iD},\,\omega_{i\sharp},\,\omega_{i0}\right) & \in \quad
\mathcal C^0(0,T_f;L^2(\Omega)) 
\times \mathcal C^0(0,T_f;H^\frac{1}{2}(\Gamma_D)) \\ 
& \times \mathcal C^0(0,T_f;H^\frac{1}{2}_{00}(\Gamma_N)^{'})
\times L^2(\Omega),\quad i=1,2
\end{alignat*}

the function $\omega_{iN\tau} $ interpolates $\omega_{iN}^m$ solutions of the following problem: 
\par  $ \omega_{iN}^0=\mathcal I_N \omega_{i0}$  in  $\Omega,$ \;
$\omega^m_{iN}=i^{\Gamma_D}_N \omega^m_i $ on  $\Gamma_D $ \quad and 
 for all  $ \eta _N\in \mathbb{Y}^\star_N$
\begin{alignat}2
 \left(\frac{\omega^m_{iN}-\omega_{iN}^{m-1}}{\tau_m},\eta _N\right)_N &+\left(\lambda_{ii}\nabla\omega^m_{iN},\nabla\eta _N\right)_N
+\left(\lambda_{ij}\nabla\omega^m_{jN},\nabla\eta _N\right)_N \notag\\
&\hspace{0cm}
=\int_\Omega h_i^m\,\eta _N d\x+
\langle
\lambda_{ii}\omega^m_{i\sharp}+\lambda_{ij}\omega_
{j\sharp}^m,\,\eta _N
\rangle_{\Gamma_N},\qquad i\neq j. \label{51}
\end{alignat}
If we set $\tilde U_{N\tau}$ the couple $(\u_{N\tau},p_{N\tau})$ and $U_{N\tau} $ the triplet $(\tilde U_{N\tau },{\vartheta}_{N\tau},\Psi_{N\tau})$, problem \eqref{50}-\eqref{51} can  equivalently be written as
\begin{eqnarray}\label{discrete_pbl}
 &\mathcal{F}_{N\tau}(U_{N\tau})=
 U_{N\tau}-\left(
\begin{array}{lll}
\mathcal{T}_{N\tau}& 0\\
0&\mathcal{L}_{N\tau}\\
\end{array}
\right)
\left(
\begin{array}{l}
\mathcal{G}_{{N\tau}}^{1}(U_{N\tau})\\
\left(\mathcal{G}_{{N\tau}}^{2}(U_{N\tau}),\omega_{iD}, \tilde\omega_{i\sharp},\omega_{i0}\right)\\
\end{array}
\right)=0 \qquad
  \end{eqnarray}
such that for all  $ \v_N\in \mathbb{X}_N $ and for all  $ \eta _N\in \mathbb{Y}_N^\star$
\begin{alignat*}{2}
\langle \mathcal{G}_{N\tau}^{1} \left(U_{N\tau}\right),\v_N \rangle&=
\left(\f({\vartheta}_{N\tau},\Psi_{N\tau}),\v_N\right)_N, \\ 
\langle \tilde{\mathcal{G}}_{N\tau}^{2}(U_{N\tau}),\eta _N \rangle
&=\left(h_i,\eta _N\right)_N
-\left((\u_{N\tau}\cdot\nabla)\omega_{iN\tau},\eta _N\right)_N, \\ 
\mbox{and} \qquad
\langle \tilde\omega_{i\sharp},\eta _N\rangle &=\left(\lambda_{ii}\omega_{i\sharp}+\lambda_{ij}\omega_{j\sharp},\eta _N\right)^{\Gamma_N}_N,\qquad i,j=1,2, i\neq j.
\end{alignat*}    
\par The next statement makes the error estimates between discrete and semi-discrete operators, 
which can be obtained by  standard arguments in spectral methods (for instance, see \cite{BMR}).
We skip the proof  and easily conclude
\begin{proposition}
For each  $N \geq 2$,  when the  solution $\big(\u^m,p^m\big)$ of problem $(\ref{darcy-semi-dis-lin})$ belongs to $H^s(\Omega)^d\times H^{s+1}(\Omega)$, for $s\geq 1$, the following error estimate holds
\begin{equation}
\begin{aligned}
& 
\|(\mathcal{T}_{\tau}-\mathcal{T}_{N\tau})(\f,\u_0)\|_{{\mathcal C^0(0,T_f;L^2(\Omega)^d)\times L^2(0,T_f;H^{1}(\Omega))}}\\
&\hspace{+1cm}\leq cN^{-s}\|\mathcal{T}_\tau(\f,g,p_b,\u_0)\|_{H^1(0,T_f;H^s(\Omega)^d)\times L^2(0,T_f;H^{s+1}(\Omega))}.
\label{T_tau-T_Ntau} 
 \end{aligned}
 \end{equation}
Moreover, if  \; $  \mathcal L_\tau(h_i,\omega_{iD},\omega_{i\sharp},\omega_{i0}) \, \mbox{belongs to} \,  H^{s+1}(\Omega),$ 
then 
\begin{equation}
\begin{aligned}
& \sum_{i=1}^2\|(\mathcal L_{\tau}- \mathcal L_{\tau N})(h_i,\omega_{iD},\omega_{i\sharp},\omega_{i0})\|_{\mathcal C^0(0,T_f;H^1(\Omega))}  \\ 
& \hspace{+2cm}\leq c
N^{-s}\sum_{i=1}^2\|\mathcal L_\tau(h_i,\omega_{iD},\omega_{i\sharp},\omega_{i0})\|_{H^1(0,T_f;H^s(\Omega))}.\label{L_tau-L_Ntau}
\end{aligned}
\end{equation}
\end{proposition}
\noindent Finally,  we are in position to give a main result of this section.
\begin{theorem}
	Assume that the data functions 
	$h_1,h_2$ in $L^2(0,T_f; H^\sigma(\Omega)), {\vartheta}_{D},\Psi_D$ in $L^2(0,T_f; H^{\sigma+\frac{1}{2}}(\Gamma_D)) $ and $
	{\vartheta}_N,\Psi_N$ in $\mathcal C^0(0,T_f;H^\sigma(\Gamma_N)),$ for $\sigma \geq \frac{1}{2}$
	 and  that $\f$ is of class $\mathcal C^2$ on $\mathbb R^2$ with bounded derivatives. 
	 
	 If  the solution $(\u,p,{\vartheta},\Psi)$ of  \eqref{var_darcy}--\eqref{var_concentration} belongs to 
\begin{alignat*}{2}
& H^2(0,T_f;H^s(\Omega)^d)\times L^2(0,T_f; H^{s+1}(\Omega)) \\ 
& \hspace{1cm}\times H^2(0,T_f;H^{s+1}(\Omega))\times H^2(0,T_f;H^{s+1}(\Omega))
\qquad \mbox{for} \; s>\frac{d}{6} 
\end{alignat*}
then there exist a neighborhood of $(\u,p,{\vartheta},\Psi)$, a positive real number $\tau_0$ and a positive integer $N_0$ such that for each $\tau$, $|\tau|\leq \tau_0$ and $N\geq N_0$, problem \eqref{discrete_pbl} admits a unique solution $(\u_{N\tau},p_{N\tau},{\vartheta}_{N\tau},\Psi_{N\tau})$ in this neighborhood. 

\noindent Moreover, there exists a nonnegative constant  $c$ such that this solution satisfies
\begin{alignat}2
& \|\u-\u_{N\tau}\|_{\mathcal C^0(0,T_f;L^3(\Omega)^d)}
\,+\, \|p-p_{N\tau}\|_{L^2(0,T_f;H^1(\Omega))} \label{EstF}\\
& \hspace{+1cm} + \, \|{\vartheta}-{\vartheta}_{N\tau}\|_{\mathcal C^0(0,T_f;H^1(\Omega))}
\,+\,\|\Psi-\Psi_{N\tau}\|_{\mathcal C^0(0,T_f;H^1(\Omega))}\notag\\
&\hspace{1.cm}
\leq c \,(|\tau|+N^{\frac{d}{6}-s})\|\u\|_{H^2(0,T_f;H^s(\Omega)^d)}
 +c ~N^{-s}\|p\|_{L^2(0,T_f;H^{s+1}(\Omega))}\notag\\
 &\hspace{.5cm}+(|\tau|+N^{-s})(\|{\vartheta}\|_{H^2(0,T_f;H^{s+1}(\Omega))}+\|\Psi\|_{H^2(0,T_f;H^{s+1}(\Omega))})\notag\\
  &\hspace{.5cm}+ c\, N^{-\sigma}\big(\|h_1\|_{L^2(0,T_f;H^\sigma(\Omega))}+\|h_2\|_{L^2(0,T_f;H^\sigma(\Omega))}+\|{\vartheta}_\sharp\|_{\mathcal C^0(0,T_f;H^{\sigma}(\Gamma_N))}\notag\\
  &\hspace{.5cm}+\|\Psi_\sharp\|_{\mathcal C^0(0,T_f;H^{\sigma}(\Gamma_N))}+\|{\vartheta}_D\|_{L^2(0,T_f;H^{\sigma+\frac{1}{2}}(\Gamma_D))}+\|\Psi_D\|_{L^2(0,T_f;H^{\sigma+\frac{1}{2}}(\Gamma_D))}\big).\notag
\end{alignat}
\end{theorem}
The proof of this theorem is based on the Brezzi-Rappaz-Raviart theorem \cite{BRR}.  So, it suffices to prove that the assumptions of the theorem of Brezzi-Rappaz-Raviart are satisfied.
We will show all of theme in Lemmas \ref{lemme1} -- \ref{lemme3}. 

We first introduce the approximation $U_{N\tau}^\diamond\,=\left(\u_{N\tau}^\diamond,p_{N\tau}^\diamond,{\vartheta}^\diamond_{N\tau},\Psi^\diamond_{N\tau}\right)$ 
of $U$ in the following discrete  space
$$ \mathcal X_{N\tau}=\mathcal C^0(0,T_f;\mathbb X_N)\times 
L^2(0,T_f; \mathbb Y_N^\diamond)\times 
\mathcal C^0(0,T_f;\mathbb Y_N)\times
 \mathcal C^0(0,T_f;\mathbb Y_N),
 $$ 
which satisfies the following estimates: 
\par for each integer numbers $\ell,s,$ $0\leq \ell\leq s$, and for each $t$, $0\leq t\leq T_f$ 
\begin{alignat}{2}
&  \|\u(\cdot,t)-\u^\diamond_{N\tau}(\cdot,t)\|_{H^\ell(\Omega)^d}\leq c N^{\ell-s}\|\u(\cdot,t)\|_{H^s(\Omega)^d},\label{u2}\\
&  \|p(\cdot,t)-p^\diamond_{N\tau}(\cdot,t)\|_{H^{\ell+1}(\Omega)}\leq cN^{\ell-s}\|p(\cdot,t)\|_{H^{s+1}(\Omega)},\label{p2}\\
&  \|{\vartheta}(\cdot,t)-{\vartheta}^\diamond_{N\tau}(\cdot,t)\|_{H^{\ell+1}(\Omega)}\leq c N^{\ell-s}\|{\vartheta}(\cdot,t)\|_{H^{s+1}(\Omega)},\label{theta3}\\
&  \|\Psi(\cdot,t)-\Psi^\diamond_{N\tau}(\cdot,t)\|_{H^{\ell+1}(\Omega)}\leq cN^{\ell-s}\|\Psi(\cdot,t)\|_{H^{s+1}(\Omega)}.\label{xi3}
 \end{alignat}
\begin{hypothesis}\label{hyp3}
We next assume that the solution $(\u,p,{\vartheta},\Psi)$ of  \eqref{var_darcy}--\eqref{var_concentration}
\begin{enumerate}[(i)]
\item belongs  for $ s>\frac{d}{6} $ to 
\begin{alignat*}{2}
& H^2(0,T_f;H^s(\Omega)^d)\times L^2(0,T_f;H^{s+1}(\Omega)) \\ 
& \hspace{+1cm}\times H^2(0,T_f;H^{s+1}(\Omega))\times H^2(0,T_f;H^{s+1}(\Omega)),
\end{alignat*} 
\item   and such that $D\mathcal{F}(U)$ is an isomorphism of  $\mathcal X$.
\end{enumerate}
\end{hypothesis}

\begin{lemma}\label{lemme1}
If $\f$ is of class $\mathcal{C}^2$  on $\mathbb{R}^2$ with bounded derivatives, there exist  nonnegative integer $N_0$, and nonnegative real $\tau_0$ such that,
 for each  $N\geq N_0$ and $\tau\leq \tau_0$, the operator $D\mathcal{F}_{N\tau}(U^\diamond_{N\tau})$  is an isomorphism of $\mathcal{X}_{N\tau}$. 
 \\ 
 Moreover, the norm of its inverse is  bounded independently  of $N$.
 \end{lemma}
\begin{proof}
We start by writing  the following expansion
 \begin{eqnarray}
 \notag & D\mathcal{F}_{N\tau}(U^\diamond_{N\tau})=D\mathcal{F}(U) 
 + \left(
 \begin{array}{lll}  
 \mathcal{T}-\mathcal{T}_{N\tau}& \quad0 \\
\quad0& \mathcal L-\mathcal L_{N\tau}\\
\end{array}
\right)  
\left(
\begin{array}{l} 
D\mathcal{G}^1(U)\\
D\mathcal{G}^2(U)\\
\end{array}
\right)
\\&\label{expansion0} \hspace{1.5cm}+\left(
\begin{array}{lll}
\mathcal{T}_{N\tau}& 0\\
0&\mathcal{L}_{N\tau}\\
\end{array}
\right)
\left(
\begin{array}{l}
D\mathcal{G}^{1}(U)-D\mathcal{G}^{1}(U_{N\tau}^\diamond)\\
D\mathcal{G}^{2}(U)-D\mathcal{G}^2(U_{N\tau}^\diamond)\\
\end{array}
\right)
\\&\hspace{2.25cm}+\left(
\begin{array}{lll}
\mathcal{T}_{N\tau}& 0\\
0&\mathcal{L}_{N\tau}\\
\end{array}
\right)
\left(
\begin{array}{l}
D\mathcal{G}^{1}(U_{N\tau}^\diamond)-D\mathcal{G}_{N\tau}^{1}(U_{N\tau}^\diamond)\\
D\mathcal{G}^{2}(U_{N\tau}^\diamond)-D\mathcal{G}_{N\tau}^{2}(U_{N\tau}^\diamond)\\
\end{array}
\right).\notag
\end{eqnarray}
Owing to assumption \ref{hyp3}-(ii), we have to prove that the last three terms in the right hand side of \eqref{expansion0} tend to $0$ when $\left(|\tau|,N\right)$ goes to $\left(0,\infty\right)$.         
              
 \noindent Let $W_{N\tau}=(\w_{N\tau},q_{N\tau}, \zeta_{N\tau}^1,\zeta_{N\tau}^2)$ be an element of unit sphere of $\mathcal X_{N\tau}$. So, we observe that
 \begin{eqnarray*}
\left(\begin{array}{l} D\mathcal{G}^1(U)\cdot W_{N\tau}\\
D\mathcal{G}^2(U)\cdot W_{N\tau}\\
\end{array}\right)=
\left(\begin{array}{l} (\partial_{\vartheta} \f({\vartheta},\Psi)\zeta^1_{N\tau}+\partial_\Psi \f({\vartheta},\Psi)\zeta^2_{N\tau},0)\\
(\w_{N\tau}\nabla w_i-\u\nabla \zeta^i_{N\tau},0,0,0)\\
\end{array}\right).
\end{eqnarray*}  
First, since $D\mathcal G^1(U)$ and $D\mathcal{G}^2(U)$ are bounded and owing to the expansion 
$\mathcal T-\mathcal T_{N\tau}=(\mathcal T-\mathcal T_\tau)+(\mathcal T_{\tau}-\mathcal T_{N\tau})$ and $\mathcal L-\mathcal L_{N\tau}=(\mathcal L-\mathcal L_\tau)+(\mathcal L_{\tau}-\mathcal L_{N\tau})$,  the convergence of the first term is a consequence of \eqref{T-Ttau}, \eqref{L-Ltau}, \eqref{T_tau-T_Ntau} and  \eqref{L_tau-L_Ntau}.  
 \begin{eqnarray*}
\lim_{N\rightarrow0}\lim_{|\tau|\rightarrow0} \left(\begin{array}{lll}  \mathcal{T}-\mathcal{T}_{N\tau}& \quad0\\
\quad0& \mathcal L-\mathcal L_{N\tau}\\
\end{array}\right)  
\left(\begin{array}{l} D\mathcal{G}^1(U)\\
D\mathcal{G}^2(U)\\
\end{array}\right)=0.
\end{eqnarray*}     
 Next, 
we will prove the convergence of three terms:
\begin{alignat*}2
&\left(\partial_{\vartheta} \f({\vartheta},\Psi)-\partial_{\vartheta} \f({\vartheta}_{N\tau}^\diamond,\Psi^\diamond_{N\tau})\right)
\zeta^1_{N\tau}+(\partial_\Psi \f({\vartheta},\Psi)-\partial_\Psi \f({\vartheta}^\diamond_{N\tau},\Psi^\diamond_{N\tau}))\zeta^2_{N\tau},\\
&\w_{N\tau}\nabla({\vartheta}-{\vartheta}^\diamond_{N\tau})-(\u-\u^\diamond_{N\tau})\nabla \zeta^1_{N\tau} \\ 
& \mbox{and } \qquad\w_{N\tau}\nabla(\Psi-\Psi^\diamond_{N\tau})-(\u-\u^\diamond_{N\tau})\nabla \zeta^2_{N\tau}.
\end{alignat*}
Due to \eqref{u2}, \eqref{theta3} and \eqref{xi3} combining with stability of linear discrete operators $\mathcal T_{N\tau}$ and $\mathcal L_{N\tau}$, we claim  
\begin{eqnarray*}
\lim_{N\rightarrow0}\lim_{|\tau|\rightarrow0} \left(
\begin{array}{lll}
\mathcal{T}_{N\tau}& 0\\
0&\mathcal{L}_{N\tau}\\
\end{array}
\right)
\left(
\begin{array}{l}
D\mathcal{G}^{1}(U)-D\mathcal{G}^{1}(U_{N\tau}^\diamond)\\
D\mathcal{G}^{2}(U)-D\mathcal{G}^2(U_{N\tau}^\diamond)\\
\end{array}
\right)
=0
\end{eqnarray*}
Finally, Owing to the definition of $D\mathcal G^1, D\mathcal G^2,D\mathcal G^1_{N\tau}$ and $ D\mathcal G^2_{N\tau},$ we check that the three following terms converge to zero, for all $\v_{N\tau}$ in $\mathbb P_N(\Omega)^d$ and $\eta _{N\tau} $ in $\mathbb P_N(\Omega)$
\begin{alignat*}2
&\int_\Omega\partial_{\vartheta} \f({\vartheta}^\diamond_{N\tau},\Psi^\diamond_{N\tau})\zeta_{N\tau}^1\cdot\v_{N\tau}\;d\x-\Big(\partial_{\vartheta} \f({\vartheta}^\diamond_{N\tau},\Psi^\diamond_{N\tau})\zeta^1_{N\tau},\v_{N\tau} \Big)_N,\\
&\int_\Omega\partial_\Psi \f({\vartheta}^\diamond_{N\tau},\Psi^\diamond_{N\tau})\zeta^2_{N\tau})\cdot\v_{N\tau}\;d\x-\Big(\partial_\Psi \f({\vartheta}^\diamond_{N\tau},\Psi^\diamond_{N\tau}))\zeta^2_{N\tau},\v_{N\tau} \Big)_N,\\
&\int_\Omega (\w_{N\tau}\nabla({\vartheta}-{\vartheta}^\diamond_{N\tau})-(\u-\u^\diamond_{N\tau})\nabla \zeta^1_{N\tau})\eta _{N\tau}d\x\\
& \hspace{+2cm}-\Big(  \w_{N\tau}\nabla({\vartheta}-{\vartheta}^\diamond_{N\tau})-(\u-\u^\diamond_{N\tau})\nabla \zeta^1_{N\tau},\eta _{N\tau}\Big)_N\\
&\int_\Omega (\w_{N\tau}\nabla(\Psi-\Psi^\diamond_{N\tau})-(\u-\u^\diamond_{N\tau})\nabla \zeta^2_{N\tau})\eta _{N\tau}d\x
\\ 
& \hspace{+2cm}- \Big(\w_{N\tau}\nabla(\Psi-\Psi^\diamond_{N\tau})-(\u-\u^\diamond_{N\tau})\nabla \zeta^2_{N\tau},\eta _{N\tau} \Big)_N.
\end{alignat*}
As we use the same arguments  to evaluate all these terms,  we  only consider  the first one.
To do so, we choose $N^*$ equal to the integer part of $\frac{N-1}{2} $ and we introduce the approximations $\f_{N^*} $  of $\partial_{\vartheta} \f({\vartheta}^\diamond_{N\tau},\Psi^\diamond_{N\tau})$  
in $\mathbb P_{N^*}(\Omega)^d$ and $\zeta^1_{N^*}$ of $\zeta^1_{N\tau}$ 
in $\mathbb P_{N^*}(\Omega)$. We point out  the identity\;
$$
\displaystyle 
\int_{\Omega} \f_{N^*}\zeta^1_{N^*}\v_{N\tau}\;d\x=\Big( \f_{N^*}\zeta^1_{N^*},\v_{N\tau}\Big)_N.
$$
Then
\begin{alignat*}2
&\int_\Omega\partial_{\vartheta} \f({\vartheta}^\diamond_{N\tau},\Psi^\diamond_{N\tau})\zeta_{N\tau}^1\cdot \v_{N\tau}d\x-\Big(\partial_{\vartheta} \f({\vartheta}^\diamond_{N\tau},\Psi^\diamond_{N\tau})\zeta^1_{N\tau},\v_{N\tau} \Big)_N\\
&\hspace{+1cm} \leq 
\; (\|\partial_{\vartheta} \f({\vartheta}^\diamond_{N\tau},\Psi^\diamond_{N\tau})\zeta_{N\tau}^1-\f_{N^*}\zeta^1_{N^*} \| \\ 
&\hspace{+2cm} + \,3^d\|\mathcal I_N(\partial_{\vartheta} \f({\vartheta}^\diamond_{N\tau},\Psi^\diamond_{N\tau})\zeta_{N\tau}^1-\f_{N^*}\zeta^1_{N^*} )\| ) \|\v_{N\tau}\|.
\end{alignat*}
 Triangular inequality and stability property of $\mathcal I_N$ 
(ses \cite[Rem. 13.5]{BerMad_HandBook}) give us
\begin{alignat*}2
&\hspace{-1.7cm}\int_\Omega\partial_{\vartheta} \f({\vartheta}^\diamond_{N\tau},\Psi^\diamond_{N\tau})\zeta_{N\tau}^1\cdot \v_{N\tau}d\x-\Big(\partial_{\vartheta} \f({\vartheta}^\diamond_{N\tau},\Psi^\diamond_{N\tau})\zeta^1_{N\tau},\v_{N\tau} \Big)_N\\
&\hspace{+1cm} \leq 
\; \big(\|\partial_{\vartheta} \f({\vartheta}^\diamond_{N\tau},\Psi^\diamond_{N\tau})-\f_{N^*}\|\|\zeta^1_{N\tau} \|_{L^\infty(\Omega)} \\ 
&\hspace{+2cm} + \, \|\zeta_{N\tau}^1-\zeta^1_{N^*} \| \|\f_{N^*}\|_{L^\infty(\Omega)^d}\big) \|\v_{N\tau}\|.
\end{alignat*}
Let us introduce the orthogonal projection operator from $L^2(\Omega)$ (or $L^2(\Omega)^d$) on $\mathbb P_{N^*}(\Omega)$ (or $\mathbb P_{N^*}(\Omega)^2$ ) (see for instance see \cite[Ch.III]{BMR}). Using the Lipschitz property of $\partial_{\vartheta} \f$ and  taking 
$$ \f_{N^*} = \Pi_{N^*}\partial_{\vartheta} \f({\vartheta},\Psi)\quad \mbox{and }\quad 
 \zeta^1_{N^*} = \Pi_{N^*}\zeta^1_{N\tau}
$$ 
%
Then, thanks to  \cite[Chap. III, Thm. 2.4]{BMR} combining with stability of $\mathcal T_{N\tau}$ and $\mathcal L_{N\tau}$, we deduce that
\begin{eqnarray*}
\lim_{N\rightarrow0}\lim_{|\tau|\rightarrow0} \left(
\begin{array}{lll}
\mathcal{T}_{N\tau}& 0\\
0&\mathcal{L}_{N\tau}\\
\end{array}
\right)
\left(
\begin{array}{l}
D\mathcal{G}^{1}(U_{N\tau}^\diamond)-D\mathcal{G}_{N\tau}^{1}(U_{N\tau}^\diamond)\\
D\mathcal{G}^{2}(U_{N\tau}^\diamond)-D\mathcal{G}_{N\tau}^{2}(U_{N\tau}^\diamond)\\
\end{array}
\right)=0.
\end{eqnarray*}
This concludes the proof.
 	\end{proof}
  \begin{lemma}\label{lemme2}
 Under the same assumptions in the previous lemma, there exist a neighborhood of $U^\diamond_{N\tau}$ in $\mathcal X_{N\tau}$ and a positive constant $c$  such that the operator $D\mathcal F_{N\tau}$ satisfies  Lipschitz property, for any $Z_N$ in this neighborhood,
 \begin{equation}
  \|D\mathcal F_{N\tau}(U^\diamond_{N\tau})-D\mathcal F_{N\tau}(Z_N)\|_{\mathcal{E}}\leq c\|U^\diamond_{N\tau}-Z_N\|_{\mathcal{X}},
 \end{equation}
 \end{lemma}
 	\begin{proof} When setting  $Z_N=(\z_N,q_N,\sigma_N,\zeta_N)$  and writing 
\begin{eqnarray*}
D\mathcal{F}_N(U^\diamond_N)-D\mathcal{F}_N(Z_N)= (U^\diamond_N-Z_N)-
\left( \begin{array}{lll}
\mathcal{T}_N& 0\\
0&\mathcal{L}_N\\
\end{array}
\right)
\left(
\begin{array}{l}
D\mathcal{G}_{N}^{1}(U_{N\tau}^\diamond)-D\mathcal{G}_{N}^{1}(Z_N)\\
D\mathcal{G}_{N}^{2}(U_{N\tau}^\diamond)-D\mathcal{G}_{N}^{2}(Z_N)\\
\end{array}
\right).
\end{eqnarray*}
We can use similar arguments in \cite[Lemma 4.5]{sarra2} to  obtain the result.
 	\end{proof} 
 \begin{lemma}\label{lemme3}
 For any $h_1,h_2$ in $ L^2(0,T_f;H^\sigma(\Omega))$, ${\vartheta}_\sharp,\Psi_\sharp$ in $L^2(0,T_f; H^\sigma(\Gamma_N))$ and ${\vartheta}_D,\Psi_D$ in $L^2(0,T_f;H^{\sigma+\frac{1}{2}}(\Gamma_D))$. 
  If $\f$ is of class $\mathcal{C}^2$ with bounded derivatives and if the Assumption \ref{hyp3} holds,
  then the following estimate is satisfied
\begin{alignat*}{2}
 & \|\mathcal{F}_{N\tau}(U^\diamond_{N\tau})\|_{\mathcal X}
 \leq c\, 
 \Big[
 \left(|\tau|+N^{\frac{d}{6}-s}\right)\, \|\u\|_{H^2(0,T_f;H^s(\Omega)^d)}\\
&\qquad+\, N^{-s}\|p\|_{L^2(0,T_f;H^{s+1}(\Omega))} +\,\left(|\tau|+N^{-s}\right)\, \sum_{i=1}^2\|\omega_i\|_{H^2(0,T_f;H^{s+1}(\Omega))}\notag \\ 
 & \qquad
 +N^{-\sigma}\sum_{i=1}^2 \big(\|h_i\|_{\mathcal{C}^0(0,T_f;H^{\sigma}(\Omega))} +\|\omega_{iD}\|_{L^2(0,T_f;H^{\sigma+\frac{1}{2}}(\Gamma_D))} 
 \\ & \hspace{+6cm}+\|\omega_{i\sharp}\|_{L^2(0,T_f;H^{\sigma}(\Gamma_N))}\big)
 \Big]
\notag
\end{alignat*}
 \end{lemma}
\begin{proof}
Since $\mathcal F(U)=0$, introducing  $\mathcal{\tilde G}_{N\tau}^2=(\mathcal G^2_{N\tau},\omega_{iD},\tilde\omega_{i\sharp},\omega_{i0})$, we write 
\begin{eqnarray*}
 \mathcal{F}_{N\tau}(U_{N\tau}^\diamond)  &=-(U-U_{N\tau       }^\diamond) +
\left(\begin{array}{lll}
\mathcal{T}-\mathcal T_{N\tau}&\quad 0\\
\quad 0 & \mathcal{L}-\mathcal L_{N\tau}\\
\end{array} \right)\left(\begin{array}{l}
\mathcal{G}^1(U)\\
\mathcal G^2(U)\\
\end{array} \right) 
\\
& \hspace{2.cm}
+ \left(\begin{array}{lll}
\mathcal T_{N\tau}& 0\\
0 & \mathcal L_{N\tau}\\
\end{array} \right)\left(\begin{array}{l}
\mathcal{G}^1(U)-\mathcal{G}^1(U_{N\tau}^\diamond)\\
\mathcal G^2(U)-\mathcal{G}^2(U_{N\tau}^\diamond)\\
\end{array} \right) 
\\ & \hspace{2.5cm} +
\left(\begin{array}{lll}
\mathcal T_{N\tau}& 0 \\
0 & \mathcal L_{N\tau}\\
\end{array} \right)\left(\begin{array}{l}
\mathcal{G}^1(U_{N\tau}^\diamond)-\mathcal{G}_{N\tau}^{1}(U_{N\tau}^\diamond)\\
\mathcal G^2(U_{N\tau}^\diamond)-\mathcal{\tilde G}_{N\tau}^{2}(U_{N\tau}^\diamond)\\
\end{array} \right) 
\\
\end{eqnarray*}
Thanks to \eqref{u2},\eqref{p2},\eqref{theta3} and \eqref{xi3}, we bound the  first term.
Concerning the second term,  we have
\begin{alignat*}2
 \|(\mathcal{T}-\mathcal T_{N\tau})(\mathcal G^1(U))\|&_{\mathcal C^0(0,T_f;L^2(\Omega)^d)\times L^2(0,T_f;H^1(\Omega))} \\ 
& \leq \, \|(\mathcal{T}-\mathcal T_{\tau})(\mathcal G^1(U))\|_{\mathcal C^0(0,T_f;L^2(\Omega)^d)\times L^2(0,T_f;H^1(\Omega))}\\
&\qquad +\, \|(\mathcal{T}_\tau-\mathcal T_{N\tau})(\mathcal G^1(U))\|_{\mathcal C^0(0,T_f;L^2(\Omega)^d)\times L^2(0,T_f;H^1(\Omega))}\\
&\leq c |\tau | \|\u\|_{H^2(0,T_f;L^2(\Omega)^d)}+
c N^{-s}\|\u_\tau\|_{H^1(0,T_f;H^s(\Omega)^d)} \\ 
& \qquad \qquad + N^{-s}\|p_\tau\|_{L^2(0,T_f; H^{s+1}(\Omega))},\\
&\leq c (|\tau|+N^{-s})\|\u\|_{H^2(0,T_f;H^s(\Omega)^d)}.
\end{alignat*}
Similarly,
\begin{alignat*}2
&\|(\mathcal{L}-\mathcal L_{N\tau})(\mathcal G^2(U))\|_{\mathcal C^0(0,T_f;H^1(\Omega))} \\
& \leq \, \|(\mathcal{L}-\mathcal L_{\tau})(\mathcal G^2(U))\|_{\mathcal C^0(0,T_f;H^1(\Omega))}+\|(\mathcal{L}_\tau-\mathcal L_{N\tau})(\mathcal G^2(U))\|_{\mathcal C^0(0,T_f;H^1(\Omega))}\\
&   \leq c  \left(|\tau |+ N^{-s}\right) \left(\|{\vartheta}\|_{H^2(0,T_f;H^{s+1}(\Omega))}+\|\Psi\|_{H^2(0,T_f;H^{s+1}(\Omega))}\right).
\end{alignat*}
Evaluating  the third term follows from the stability property of operators $\mathcal T_{N\tau}$ and $\mathcal L_{N\tau}$   combining with  \eqref{u2}, \eqref{theta3} and \eqref{xi3}:
\begin{eqnarray*}
&\left\|\left(\begin{array}{lll}
\mathcal T_{N\tau}& 0\\
0 & \mathcal L_{N\tau}\\
\end{array}
 \right)\left(\begin{array}{l}
\mathcal{G}^1(U)-\mathcal{G}^1(U_{N\tau}^\diamond)\\
\mathcal G^2(U)-\mathcal{G}^2(U_{N\tau}^\diamond)\\
\end{array} \right)\right\|_{\mathcal X} \\
&
\leq c \big(N^{\frac{d}{6}-s}\|\u\|_{\mathcal C^0(0,T_f;H^s(\Omega)^d)} \\
& \hspace{+4cm}+ N^{-s} (\|{\vartheta}\|_{\mathcal C^0(0,T_f;H^{s+1}(\Omega))}+\|\Psi\|_{\mathcal C^0(0,T_f;H^{s+1}(\Omega))}) \big).
\end{eqnarray*}
Finally, proving the estimate for the fourth term is
obtained from stability of $\mathcal T_{N\tau}$ and $\mathcal L_{N\tau}$ and
by using the standard arguments of numerical integration error.
	\end{proof}
	\begin{remark}
\begin{enumerate}[(1)]
	\item  All assumptions of the Brezzi-Rappaz-Raviart theorem  are obtained in Lemmas \ref{lemme1}-- \ref{lemme3}.
		\item  the error behaves  $|\tau|+N^{\frac{d}{6}-s}$   optimal in time and nearly optimal in space.
\end{enumerate}
\end{remark}
\section{Numerical results}
\par We report in this section several numerical tests,
the aim being to evaluate the performance of the spectral discretization in two and
three space dimensions. All the computations have been performed on the code
FreeFEM3D--spectral version, see  \cite{delpino1} and \cite{driss2}.  
However, before we do it, we wish to explain how to solve the coupled problem.
Indeed, it is clear that when using   the implicit Euler's scheme in time and spectral method in space, 
the obtained discrete problem is still  coupled and  nonlinear. To solve it numerically, 
we propose the following linear iterative process: 
\par At each time step  $t_{m}$: knowing the solution
$\left( \u^{m-1},p^{m-1},{\vartheta}^{m-1}, \Psi^{m-1} \right) $ at time $t_{m-1}$. For a given tolerance $\varepsilon >0$, compute  :
\begin{enumerate} 
\item {\bf Step 1:\, Initialization:} 
$$\left( \u^{m}_{0},p^{m}_{0},{\vartheta}^{m}_{0}, \Psi^{m}_{0} \right)  = 
\left( \u^{m-1},p^{m-1},{\vartheta}^{m-1}, \Psi^{m-1} \right). $$
\item {\bf Step 2:\, Darcy Eq.:} \\ 
 For a nonnegative integer $k>0$, knowing $\left({\vartheta}^{m}_{k}, \Psi^{m}_{k}\right)$,
  \\  we compute  $(\u^{m}_{k+1},p^{m}_{k+1})$ solution of 
\begin{alignat*}2
& \frac{\u^{m}_{k+1}-\u^{m-1}}{\tau_m}+\alpha \u^{m}_{k+1}+\nabla p^{m}_{k+1}= \f({\vartheta}^{m}_{k},\Psi^{m}_{k}),\\
&\nabla\cdot \u^{m}_{k+1}=0.
\end{alignat*}
\item {\bf Step 3:\, Heat Eq.:}\quad  Find  ${\vartheta}^{m}_{k+1}$ solution of 
\begin{alignat*}2
& \frac{{\vartheta}^{m}_{k+1}-{\vartheta}^{m-1}}{\tau_m}+ (\u^{m}_{k+1}\cdot\nabla) {\vartheta}^{m}_{k+1}-\nabla\cdot(\lambda_{11}\nabla {\vartheta}^{m}_{k+1})-\nabla\cdot(\lambda_{12}\nabla \Psi^{m}_{k})= h_1^{m},
\end{alignat*}
\item {\bf Step 4:\, Concentration Eq.:}\quad Finally, we compute  $\Psi^{m}_{k+1}$ solution of 
\begin{alignat*}2
&\frac{\Psi^{m}_{k+1}-\Psi^{m-1}}{\tau_m}+ (\u^{m}_{k+1}\cdot\nabla) \Psi^{m}_{k+1}-\nabla\cdot(\lambda_{22}\nabla \Psi^{m}_{k+1})-\nabla\cdot(\lambda_{21}\nabla {\vartheta}^{m}_{k+1})= h_2^{m}.
\end{alignat*}
\item {\bf Step 5:\, Goto Step 2 until:}\quad  
\begin{alignat*}{2}
&  \| \u^{m}_{k+1} - \u^{m}_{k}\|^2 
+ \|p^{m}_{k+1}  -p^{m}_{k}  \|_{H^1(\Omega)}^2 \\
&\hspace{+1cm} + 
\|T^{m}_{k+1}  -T^{m}_{k}  \|_{H^1(\Omega)}^2
+ \|C^{m}_{k+1}  -C^{m}_{k+1}  \|_{H^1(\Omega)}^2 \,\le \varepsilon.
\end{alignat*}
\end{enumerate}
Note that, the obtained linear systems are solved using a preconditioned GMRES, see Saad \cite{Saa2003}.
\subsection{Time accuracy}
To confirm our theoretical results, we are interested to calculate the error due to the  time discretization. 
In order to do so, we construct the three-dimensional problem where the exact solution is given in the unite cube $] 0.1 [^3$ by
\begin{alignat}2\label{solution-test-4}
\begin{array}{rcl}
 u_1&=& \cos(t)y- \sin(\pi+t)z^2,\\
 u_2 &=& \sin(t)(x-1)+\cos(\pi t),\\
u_3&=&-2tx,
\end{array}
\hspace{1cm}
\begin{array}{rcl}
 p &=&  \sin(t) x+\cos(t)(y+z^2),\\
 T &=& \cos(t)(x^2+2 y^2 - z),\\
C &=& \sin(t)(-x + y^3). 
\end{array}
\end{alignat}
We choose the coefficients $\alpha$ and $\lambda_{ij}$, $i,j = 1, 2$ as follow
\begin{alignat*}2
&\alpha = T^2+ C^2+ 2,\\
&\lambda_{11}= T+ C + 10, \;\qquad \lambda_{12}=0,\\
&\lambda_{22}= T^2 + C^2 + 2, \qquad \lambda_{21}=T+C.
\end{alignat*}
Hence, the suitable forcing functions $\f, h_1,h_2$ and  boundary conditions  are obtained using  these exact solutions in  
our system.
We recall that all  parameters $\alpha$ and  $(\lambda_{ij})_{1\le i,j \le2}$ are replaced by their Lagrange interpolates.
\par We perform  several simulations by dividing successively the time step $ \delta t$ (starting by $\delta t =0.1$) 
by $ 2$   and  taking  a fixed polynomial degree $N = 5$. 
Note that, at each time,  the exact solution  is polynomial function with degree less than $3$. 
Consequently, the  space error is exactly equal to $0$ and then, only the temporal error will be observed. 
\begin{alignat*}{2}
\mathcal E_{\delta t}  
=  \left( \| \u - \u_N \|^2 + \|p -p_N \|_{H^1(\Omega)}^2  + 
\|T -T_N \|_{H^1(\Omega)}^2+ \|C -C_N \|_{H^1(\Omega)}^2\right)^{\frac{1}{2}}.
\end{alignat*}
Finally, we calculate 
$  \mathcal O_{\frac{\delta t}{2}} = \displaystyle \frac{ log( \frac{\mathcal E_{\delta t}} {\mathcal E_{\frac{\delta t}{2}}})}{log(2)}$ 
which is the desired convergence rate.
\par \par In Table \ref{table4}, we plot separately  the $L^2-$error of the velocity, the $H^1-$error of the temperature, concentration and pressure  and the total error $\mathcal E_{\delta t}$
between the numerical solution and the exact solution at final time $T_f = 1$.
As we can see,  the experimental convergence  rate is close to $ 1 $, which is 
 in concordance with a priori error estimate obtained above, when the backward Euler time differentiation is used. 
\begin{table}[h] \caption{Time accuracy: Convergence rate  when $N=5$}\label{table4}

  \begin{center}
 \begin{tabular}[http]{lclclclclcl}
  \hline   &&& & & &  \\  
$\delta t $    & $ \| \u -\u_{N \tau}\|_{L^2}  $      
& $\| p- p_{N\tau} \|_{H^1} $ &  $ \| T- T_{N\tau} \|_{H^1}$   & $ \| C- C_{N\tau} \|_{H^1} $ & $\mathcal E_{\delta t}$  & $\mathcal O_{\frac{\delta t}{2}}$  
  \\  & & & & & & 
       \\   \hline  
  $\frac{1}{10}$ & $0.0857$    &   0.0129   &   0.0014         & 0.0023    &  0.0867  & -----
        \\     & & & & & &     \\     \hline 
 $\frac{1}{20}$ &  0.0439      &  0.0066&   0.0007         & 0.0012   &  0.04443  & 0.9655  
         \\   & & & &&&\\   \hline
 $\frac{1}{40}$ &  0.0222     &   0.0032 &   0.0003          & 0.0008   &  0.0224  &0.9822
           \\  & & & &&& \\   \hline 
 $\frac{1}{80}$ &  0.0112     &  0.0018 &   0.0002        & 0.0006   &  0.0113  &0.9855
          \\ &  &&& & & \\     \hline
 $\frac{1}{160}$ & 0.0056     &  0.0011  &  0.0001& 0.0005   &  0.0057 &     0.9787
          \\  & & &  &&& \\ \hline
  \end{tabular}
  \end{center}
  \end{table}

\subsection{Space accuracy}
In this test, the rate of convergence with respect to polynomial degree $N$ 
for  $\left(\u,p,{\vartheta},\Psi \right)$ in the $L^2$-norm and $H^1$-norm have been tested numerically on the square $ ] - 1.1 [^ 2 $. 
We choose  two-dimensional analytic  solutions ( with the appropriate source terms) which are  affine functions with respect to $t$. 
So that the error in time is zero :
\begin{alignat}2
& u_1= -t\sin(\pi  x) \cos(\pi y) +t+1,\quad u_2 = t \cos(\pi x)\sin(\pi y)+2t+1,\notag\\
& p = \frac{-1}{\pi} \sin(\pi x)\cos(\pi y),\label{sol-analytique-4}\\
& T = t(2\cos(\pi x)\sin(\pi y) +1),\quad C = t \sin(\pi x)\cos(\pi y)\sin(\pi(x+y))+t-1. \notag
\end{alignat}
The conductivity and permeability coefficients are given as follows
\begin{alignat*}2
& \alpha(T,C) = \frac{1}{T^2+C^2+1},\quad 
\lambda_{ii}=\lambda_{21}= T^2+C^2+2,\quad i=1,2\quad{\rm{and}}\quad \lambda_{12}=0.
\end{alignat*}
We fixed the time step $ \delta t $ equal to $ 0.1$. 
We illustrate the behavior of the error between the exact solution and the
discrete solution versus the polynomial degree $N$ which is varying between
$N = 5$ and $N = 25$. 
The spectral convergence can be easily observed in Figure \ref{convergence-D-C-L2-H1} where we  present the $L^2-$error  of velocity and $H^1-$error of others in semi-logarithmic scales.
\begin{figure}
\centering	
\begin{tikzpicture}[scale=1.3, every mark/.append style={scale=1}]
\begin{semilogyaxis}[ylabel =Logarithmic error , xlabel= N, xmax=28,ymax=30,legend style={cells={anchor=west}, legend pos=north east}]
\addplot+[color=blue,xscale=1,yscale=1] table [x=N,y=ErL2U ] {courbe-Chp32.data};
\addplot+[color=red,xscale=1,yscale=1] table [x=N,y=ErH1P] {courbe-Chp32.data};
\addplot+[color=green,xscale=1,yscale=1] table [x=N,y=ErH1T] {courbe-Chp32.data};
\addplot+[color=violet,xscale=1,yscale=1] table [x=N,y=ErH1C] {courbe-Chp32.data};
\legend{velocity,pressure, temperature,concentration}
at={(0.98,0.98)}
\end{semilogyaxis}
\end{tikzpicture}
\caption{ The errors versus polynomial degree $N$: $L^2$-error of velocity and $H^1$-error of others} \label{convergence-D-C-L2-H1}
\end{figure}
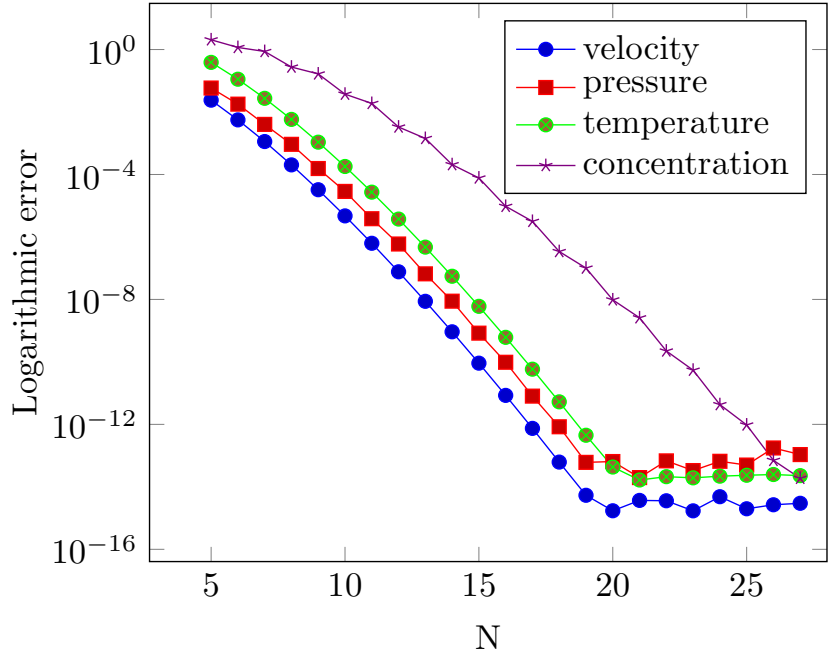

 \section{Conclusion}
In this paper, we have analyzed a model for the coupling of
the heat and mass equations with Darcy's equations for an incompressible  fluid, where
all diffusion parameters (thermal and mass diffusivity and Dufour and Soret coefficients)  are variables. We have proved that this  problem admits at least a solution in suitable spaces. To approximate its solution, we have used the high polynomial approximation, namely Spectral method and we have  proved its well-posedness. Thanks to the Brezzi-Rappaz-Raviart theorem, the a priori analysis is proved and the optimal errors are obtained.  Finally numerical tests confirm these theoretical findings. 

\bibliographystyle{plain}
\bibliography{references}

\begin{thebibliography}{10}

\bibitem{Adams}
R.A. Adams and J.~Fournier.
\newblock {\em {Sobolev Spaces}}.
\newblock Academic Press, 2003.

\bibitem{alam}
M.~S. Alam and M.~M. Rahman.
\newblock Dufour and {S}oret effects on mixed convection flow past a vertical
  porous flat plate with variable suction.
\newblock {\em Nonlinear Analysis: Modelling and Control}, 11(1):3--12, 2006.

\bibitem{ABDG}
C.~Amrouche, C.~Bernardi, M.~Dauge, and V.~Girault.
\newblock {Vector potentials in three-dimensional non-smooth domains}.
\newblock {\em Math. Methods Appl. Sci.}, 21(9):823--864, 1998.

\bibitem{anghel}
M.~Anghel, H.~S. Takhar, and I.~Pop.
\newblock Dufour and {S}oret effects on free convection boundary-layer over a
  vertical surface embedded in a porous medium.
\newblock {\em Studia Universitatis Babes-Bolyai, Mathematica}, 45(4):11--21,
  2000.

\bibitem{bergirraj}
C.~Bernardi, V.~Girault, and K.~Rajagopal.
\newblock {Discretization of an unsteady flow through porous solid modeled by
  {D}arcy's equations}.
\newblock {\em Math. Models Methods Appl. Sci.}, 18(12):2087--2123, 2008.

\bibitem{sarra2}
C.~Bernardi, S.~Maarouf, and D.~Yakoubi.
\newblock {Spectral discretization of Darcy's equations coupled with the heat
  equation}.
\newblock {\em IMA Journal of Numerical Analysis}, 36(3):1193--1216, 2015.

\bibitem{BerMad_HandBook}
C.~Bernardi and Y.~Maday.
\newblock {\it Spectral Methods}.
\newblock In {\em {Handbook of Numerical Analysis, {V}ol. {V}, $209-485$}}.
  North-Holland, Amsterdam, 1997.

\bibitem{BMR}
C.~Bernardi, Y.~Maday, and F.~Rapetti.
\newblock {\em {Discr{\'e}tisations variationnelles de probl{\`e}mes aux
  limites elliptiques}}, volume~45 of {\em {Math{\'e}matiques \&
  Applications}}.
\newblock Springer-Verlag, Paris, 2004.

\bibitem{BRR}
F.~Brezzi, J.~Rappaz, and P.-A. Raviart.
\newblock {Finite-dimensional approximation of nonlinear problems. {I}.
  {B}ranches of nonsingular solutions}.
\newblock {\em Numer. Math.}, 36(1):1--25, 1980/81.

\bibitem{delpino1}
S.~{Del Pino} and O.~Pironneau.
\newblock {A fictitious domain based on general PDE's solvers}.
\newblock In {\em {K. Morgan (Ed.), Proc. ECCOMAS 2001, Sept Swansea}}. Wiley,
  2002.

\bibitem{elarabawy}
H.~A.~M. El-Arabawy.
\newblock Soret and {D}ufour effects on natural convection flow past a vertical
  surface in a porous medium with variable surface temperature.
\newblock {\em Journal of Mathematics and Statistics}, 5(3):190--198, 2009.

\bibitem{girault}
V.~Girault and P.-A. Raviart.
\newblock {\em {Finite Element Methods for {N}avier-{S}tokes Equations, Theory
  and Algorithms}}, volume~5 of {\em {Springer Series in Computational
  Mathematics}}.
\newblock Springer-Verlag, Berlin, 1986.

\bibitem{ingham}
D.~B. Ingham and I.~Pop.
\newblock {\em Transport Phenomena in Porous Media III}, volume~3.
\newblock Elsevier, Oxford, UK, 2005.

\bibitem{lions}
J.-L. Lions and E.~Magenes.
\newblock {\em {Probl{\`e}mes aux limites non homog{\`e}nes et applications,
  {V}ol. 1}}.
\newblock Dunod, Paris, 1968.

\bibitem{moorthy}
M.~B.~K. Moorthy and K.~Senthilvadivu.
\newblock Soret and {D}ufour effects on natural convection flow past a vertical
  surface in a porous medium with variable viscosity.
\newblock {\em J. Appl. Math.}, pages 06-- 15, 2012.

\bibitem{narayana}
P.~A.~L. Narayana.
\newblock Soret and {D}ufour effects on free convection heat and mass transfer
  in a doubly stratified darcy porous medium.
\newblock {\em Journal of Porous Media}, 10(6), 2007.

\bibitem{NieldBejan}
D.~A. Nield and A.~Bejan.
\newblock {\em {Convection in Porous Media}}.
\newblock Springer-Verlag, New York, second edition, 1999.

\bibitem{pop}
I.~Pop and D.~B. Ingham.
\newblock {\em Convective Heat Transfer: Mathematical and Computational
  Modelling of Viscous Fluids and Porous Media}.
\newblock Elsevier, 2001.

\bibitem{postelnicu1}
A.~Postelnicu.
\newblock Heat and mass transfer by natural convection at a stagnation point in
  a porous medium considering {S}oret and {D}ufour effects.
\newblock {\em Heat and mass transfer}, 46(8-9):831--840, 2010.

\bibitem{Saa2003}
Y.~Saad.
\newblock {\em {Iterative Methods for Sparse Linear Systems}}.
\newblock SIAM, Philadelphia, 2003.

\bibitem{vadasz}
P.~Vad{\'a}sz.
\newblock {\em Emerging Topics in Heat and Mass Transfer in Porous Media: From
  Bioengineering and Microelectronics to Nanotechnology}, volume~22.
\newblock Springer Science \& Business Media, 2008.

\bibitem{vafai}
K.~Vafai.
\newblock {\em Handbook of Porous Media}.
\newblock Taylor and Francis, New York, USA, 2005.

\bibitem{driss2}
D.~Yakoubi.
\newblock {\em {Analysis and Development of New Algorithms in Spectral
  Methods}}.
\newblock Thesis, {Universit{\'e} Pierre et Marie Curie - Paris VI}, December
  2007.

\end{thebibliography}
\newpage

\end{document}